\documentclass[12pt]{amsart} 

\usepackage[active]{srcltx}

\newtheorem{definition}{Definition}[section]

\newtheorem{remark}[definition]{Remark}
\newtheorem{example}[definition]{Example}

\newtheorem{algorithm}[definition]{Algorithm}
\newtheorem{lemma}[definition]{Lemma}
\newtheorem{proposition}[definition]{Proposition}
\newtheorem{theorem}[definition]{Theorem}
\newtheorem{corollary}[definition]{Corollary}

\def\Q{{\mathbb{Q}}}

\def\N{{\mathbb{N}}}
\def\cs{{\mathcal{s}}}

\def\R{{\mathbb{R}}}

\def\bfalpha{{\mathbf{\alpha}}}
\def\bfbeta{{\mathbf{\beta}}}

\def\bX{{\mathbf{X}}}
\def\bg{{\mathbf{g}}}

\def\bfr{{\mathbf{r}}}

\def\cS{{\mathcal{S}}}

\def\cG{{\mathcal{G}}}
\def\cF{{\mathcal{F}}}
\def\cM{{\mathcal{M}}}

\def\Z{{\mathbb{Z}}}

\def\bfgamma{{\bf\gamma}}

\def\bfv{{\mathbf{v}}}
\def\bfx{{\mathbf{x}}}

\def\bfS{{\mathbf{S}}}
\def\bfY{{\mathbf{Y}}}

\def\rmR{{\mathrm{R}}}

\def\fm{{\mathfrak{m}}}
\def\C{{\mathrm{C}}}

\def\Y{{\underline{Y}}}

\def\bfc{{\bf{c}}}
\def\bfp{{\bf{p}}}

\def\cB{{\mathcal{B}}}
\def\bfa{{\bf{a}}}
\def\bfb{{\bf{b}}}
\def\bfs{{\bf{s}}}
\def\bff{{\bf{f}}}
\def\bfu{{\bf{u}}}
\def\bfw{{\bf{w}}}
\def\bfz{{\bf{z}}}
\def\cH{{\mathcal{H}}}

\def\bfI{{\bf{I}}}

\begin{document}

\title[Resolutions of fan algebras of principal ideals]{On the resolution of fan algebras of principal ideals over a Noetherian ring}

\author{Teresa Cortadellas Ben\'itez}
\address{Facultat de Educaci\'o, Universitat de Barcelona.
Passeig de la Vall d'Hebron 171,
08035 Barcelona, Spain}
\email{terecortadellas@ub.edu}

\author{Carlos D'Andrea}
\address{Universitat de Barcelona, Departament de Matem\`atiques i Inform\`atica,
 Universitat de Barcelona (UB),
 Gran Via de les Corts Catalanes 585,
 08007 Barcelona,
 Spain} 
\email{cdandrea@ub.edu}
\urladdr{http://atlas.mat.ub.es/personals/dandrea}

\author{Florian Enescu}
\address{Department of Mathematics and Statistics, Georgia State University, Atlanta, GA 30303 USA}
\email{fenescu@gsu.edu}
\urladdr{http://www2.gsu.edu/~matfxe/}
\thanks{T.~Cortadellas is supported by the Spanish MEC research project MTM2013-40775-P, C.~D'Andrea is supported by the Spanish MINECO/FEDER
research project MTM 2015-65361-P and the ``Mar\'ia de Maeztu'' Programme for Units of Excellence in R\&D (MDM-2014-0445). }
\date{\today}


\begin{abstract}
We construct explicitly a resolution of a fan algebra of principal ideals over a Noetherian ring  for the case when the fan is a proper rational cone in the plane.  Under some mild conditions on the initial data, we show that this resolution is minimal.
\end{abstract}

\maketitle
\section{Introduction}\label{introduction}

Let $\rmR$ be a commutative ring, $m,\,n \geq 1$ integers,$\,p_1,\ldots, p_n\in\rmR,\,\cF$ a fan in $\R^m$ such that its support $|\cF|$ is a convex set, and $f_1,\ldots, f_n$ fan linear maps on $\cF$ (see Definition~\ref{flm} for more details). With all this data,  one can consider the following algebra 
\begin{equation}\label{pfalgebra}
\cB_{\bff,\cF,\bfp} =\sum_{\bfv\in|\cF|\cap\Z^m}\langle p_1^{f_1(\bfv)}\cdot\ldots\cdot p_n^{f_n(\bfv)}\rangle\bfx^\bfv\subset \rmR[x_1^{\pm1},x_2^{\pm1},\ldots, x_m^{\pm1}],
\end{equation}
where $x_1,\ldots, x_m$ are indeterminates, $\bff=(f_1,\ldots, f_n),\,\bfp=(p_1,\ldots, p_n)$, $\bfx^{\bfv} = x_1^{v_1}\cdots x_m^{v_m}$ and $|\cF|\subset\R^m$ is the support of the fan.  The $\rmR$-algebra $\cB_{\bff,\cF,\bfp}$ is called {\em the fan algebra} associated to the data $(\bff, \cF, \bfp).$ It is a natural combinatorial object associated to a fan. In this paper, we will describe the free resolution of $\cB_{\bff,\cF,\bfp}$ as an $\rmR$-algebra, when $m=2,\, \rmR$ is Noetherian, and $|\cF|$ is a proper rational cone.
\par Fan algebras have been introduced in~\cite{mal13a, mal13b} (see also~\cite{EM14}), but here we extend the definition further. In the aforementioned papers, it has been noted that they generalize intersection algebras in the case of principal ideals, an interesting class of algebras in its own right. The finite generation of intersection algebras has applications to the lengths of Tors of quotients by powers of ideals, as shown in~\cite{F}, and to the asymptotic growth of powers of ideals, as shown in~\cite{CES}.  

Before we state the main result of this paper, we will introduce the necessary definitions and put the fan algebra in context.

\begin{definition}\label{1.2}  A {\rm convex polyhedral cone} in $\R^m$ is a convex cone generated by a finite subset $S \subset \mathbb{R}^m$, i.e. a set of the form
$$\sigma = \left\{ \sum_{u \in S} \lambda_u \cdot u : \lambda_u \geq 0\right\}.$$
A {\rm supporting hyperplane} for $\sigma$ is a hyperplane $H$ in $\R^m$ containing the origin,  and such that $\sigma$ is contained in one of the half-spaces determined by $H$. The intersection between a supporting hyperplane $H$  and $\sigma$ is by definition a {\rm face} of $\sigma.$ The convex polyhedral cone $\sigma$ is called {\rm rational} if $S$ can be taken from $\Z^m$. A {\rm strongly convex polyhedral cone}  is a convex polyhedral cone $\sigma$ with $\sigma \cap -\sigma = 0$. In what follows, a convex rational polyhedral cone  which is not either the whole space nor contained in a line , will simply be called a proper rational cone.
\end{definition}

We now recall the notion of a fan, commonly used in toric geometry, that is of central importance in this paper.

\begin{definition}
A {\rm fan} $\cF$ in $\R^m$ is a finite collection of strongly rational cones $\sigma$ such that 

\begin{enumerate}
\item for all $\sigma \in \cF$, each face of $\sigma$ is also in $\cF$.

\item for all $\sigma_1, \sigma_2$ in $\cF$, the intersection $\sigma_1 \cap \sigma_2$ is a face of each.

\end{enumerate}
The {\rm support} of $\cF$ is the union of all $\sigma$ in $\cF$. We will denote this by $|\cF|$. In this paper, we make the additional assumption that the support $|\cF|$ is also a convex cone.

\end{definition}

\begin{example}
\label{f23}
A simple example of a fan is $\{ \sigma_1, \sigma_2, \tau_0, \tau_1, \tau_2, \tau_3\}$ where $$\sigma_1=\{ (x,y) \in {(\R_{\geq0}})^2: 3x \geq 2y \},\ \sigma_2 =  \{ (x,y) \in ({\R_{\geq0})^2}: 3x \leq 2y \},$$

$$\tau_0=\{ (x,0) \in ({\R_{\geq0}})^2\}, \  \tau_1=\{ (x,y) \in ({\R_{\geq0}})^2: 3x = 2y\}, \ \tau_2=\{ (0,y) \in ({\R_{\geq0}})^2\}, \ \tau_3=\{(0,0)\}. $$ Its support is $(\R_{\geq0})^2$.

\end{example}

\begin{definition}
\label{flm}
Let $\cF$ be a fan in $\R^m$. A function  $f : |\cF| \cap \Z^m \to \N$ is called {\rm fan linear} if $f$ is $\N$-linear on each face of $\cF$ and subadditive on the support of the fan, i.e. $f(u+v) \leq f(u)+f(v)$ for all $u, v \in  |\cF| \cap \Z^m$.

\end{definition}

\begin{example} 
\label{ff23}
For  the fan $\cF$ from Example~\ref{f23}, let $f(x,y)= \max(3x,2y)$, for all $(x, y) \in \N^2 = |\cF | \cap \Z^2$. This obviously defines a fan linear map on $\cF.$ 

\end{example}

With these definitions in place, we can define the concept that is central to our paper.

\begin{definition}
Let $n, m \geq 1$ be integers, $\cF$ a fan in $\R^m$, $\bff=(_1, \ldots, f_n)$ a collection of fan linear maps on $\cF$, and $\bfI = (I_1, \ldots, I_n)$ a collection of ideals in a commutative ring $R$.
 
The {\em fan algebra} associated to the data $(\bff, \cF,\,\bfI)$ is defined as

\begin{equation}\label{falgebra}
\cB_{\bff,\cF,\bfI}:=\sum_{\bfv\in|\cF|\cap\Z^m}I_1^{f_1(\bfv)}\ldots \,I_n^{f_n(\bfv)} \bfx^\bfv\subset \rmR[x_1^{\pm1},x_2^{\pm1},\ldots, x_m^{\pm1}],
\end{equation}
where $\bfx=(x_1,\ldots, x_m)$ are indeterminates.
\end{definition}

In the case $I_j$ is the principal ideal $\langle p_j\rangle,$ for all $\,j=1,\ldots, n,$ we have that $\cB_{\bff,\cF,\bfI}$ is the fan algebra $\cB_{\bff,\cF,\bfp}$ defined in \eqref{pfalgebra}.

\begin{remark}{\rm We summarize here some special cases of the definition, putting it in perspective and showing the relevance of this concept in the generality presented here.
\label{ff} $^{}$ 
\begin{enumerate}
\item If the support of $\cF$ is $\R _{\geq 0}^m$, we recover the definition given by Malec in~\cite{mal13b}.
\item  If $|\cF|=(\R_{\geq0})^m$ and $f_i=0,\,i=1,\ldots, n,$ it is easy to see that $\cB_{\bff,\cF,\bfI}= \rmR[x_1,\ldots, x_m],$ the ring of polynomials in $m$ variables. If instead we have that $|\cF|=\R^m$ and the $f_i's$ identically zero as before, we obtain  $\rmR[x_1^{\pm1},\ldots, x_m^{\pm1}],$ the ring of Laurent polynomials in $m$ variables. 
\item Let $I_1, \ldots, I_n$ be ideals of $R$. Their {\em intersection algebra} is defined as $$\cB_R(I_1,\ldots, I_n):=\sum_{r_1, \ldots, r_n \in\N}(I_1^{r_1}\cap \ldots \cap I_n^{r_n})x_1^{r_1}\cdots x_n^{r_n}\subset \rmR[x_1,x_2, \ldots, x_n].$$Intersection algebras of principal ideals in a UFD are fan algebras, as it was shown in \cite{mal13a}. A particularly interesting case is that of two principal monomial ideals in $k[p_1, \ldots, p_n]$, where $k$ is a field and $p_1, \ldots, p_n$ indeterminates. More specifically let $\bfa =(a_1, \ldots,a_n), \bfb=(b_1, \ldots, b_n)$ be two nonnegative integer vectors and $I_1= \langle \bfp^\bfa\rangle, I_2=\langle\bfp^\bfb\rangle.$ Then $I_1^{r} \cap I_2^s =\langle p_1^{\max(a_1 r , b_1 s)} \cdots p_n^{\max(a_n r, b_n s)}\rangle$ and
$$\cB_\rmR (I_1, I_2) = \sum_{r \geq 0, s\geq 0} \langle p_1^{\max(a_1 r , b_1 s)} \cdots p_n^{\max(a_n r, b_n s)}\rangle x_1^rx_2^s.$$ We will denote this algebra as $\cB(\bfa, \bfb)$. This algebra is toric, normal, Cohen-Macaulay and of dimension $n+2$, as shown in~\cite{mal13b, mal13a, EM14}.

\item Set $m=n,\,|\cF|=(\R_{\geq0})^n,$  and $\bfI = (I_1, \ldots, I_n)$ ideals in $R$ and $f_i:(\R_{\geq0})^n \cap \Z^n\to \N$ the projection on the $i$th coordinate, $i=1, \ldots, n$. Then 
$$\cB_{\bff,\cF,\bfI} = \sum_{(k_1, \dots, k_n) \in \N^n} I_1^{k_1} \cdots I_n^{k_n} x_1^{k_1} \cdots x_n^{k_n}$$ is the multi-Rees algebra of the ideals $I_1, \ldots, I_n$ in $\rmR$.
\item
Let $I$ be an ideal in $\rmR$. Then, for the fan and the fan linear map in Example~\ref{f23} and Example~\ref{ff23}, respectively, we have
\begin{equation}\label{BRp}
\cB_{\bff, \cF, I} = \sum_{(r,s) \in \N^2} I^{\max(3r, 2s)} x_1^r x_2^s,
\end{equation} which is the same as $\cB _\rmR ((p^3), (p^2))$, when $\rmR$ is an UFD and $p$ a prime element of $\rmR$. We will write $\cB_{\rm R, p} (3,2)$ to denote this $\rmR$-algebra.

\end{enumerate}
}
\end{remark}

\par Let $\cB_{\bff,\cF,\bfI}$ be a fan algebra.  When $\rmR$ is Noetherian, this $\rmR$-algebra is finitely generated. Indeed, a set of generators of this algebra over $\rmR$ was given in \cite[Theorem 2.3.7]{mal13a} (see also \cite{mal13b,EM14}) for the case $\rmR$ being a domain  and $|\cF|=(\R_{\geq0})^m.$ It is easy to see that 
such a proof can be easily extended to the case of a fan $\cF$ with the properties given above, and any Noetherian ring $\rmR,$ as we recall here:  denote with $C_1,\ldots, C_\ell$ the maximal cones of $\cF,$ and  for $i=1,\ldots,\ell,$ set $Q_i=C_i\cap\Z^m.$ This {\it pointed monoid} has a unique  Hilbert Basis, which we denote with:
\begin{equation}\label{fine}
\cH_{Q_i}=\{\bfv_{i1},\ldots, \bfv_{ig_i}\}\subset\Z^m.
\end{equation}
As $\rmR$ is Noetherian, for each $\bfv_{ij}$, the ideal $I_1^{f_1(\bfv_{ij})}\cdot\ldots\cdot I_n^{f_n(\bfv_{ij})}$ is finitely generated, i.e.
$$I_1^{f_1(\bfv_{ij})}\cdot\ldots\cdot I_n^{f_n(\bfv_{ij})}=\langle r_{ij1},\ldots, r_{ijh_{ij}}\rangle.
$$
The next result follows straightforwardly from Theorem 2.3.7 in \cite{mal13a}.
\begin{theorem}\label{mt1}
With notations and assumptions as above, $\cB_{\bff,\cF,\bfI}$ is generated  as an algebra over $\rmR$ by the set
\begin{equation}\label{sg}
\{r_{ijh}\, \bfx^{\bfv_{ij}}; \,i=1,\ldots,\ell,\,j=1,\ldots, g_i,\,h=1,\ldots, h_{ij}\}.
\end{equation}
\end{theorem}
In general \eqref{sg} is far from being a minimal set of generators of the fan algebra, although in some cases -like in the intersection algebra of principal ideals, see \cite{mal13b}- it has been shown that it is. It is also of interest to compute a  whole resolution of $\cB_{\bff,\cF,\bfI}$ as an $\rmR$-algebra, but very little is known at the present about this ring in general. In this text, we describe completely the case $m=2,$ with all  the ideals being principal, i.e. when $I_j=\langle p_j\rangle$ for a suitable $p_j\in\rmR,$ and  $|\cF|\neq\R^2$ (and hence contained in a half-plane, since the support of the fan is assumed to be convex). This situation includes the intersection algebras of principal ideals, hence we generalize and extend the results in \cite{mal13b}. Our main result given in Theorem \ref{mresult} describes a  resolution of this $\rmR$-algebra. If $\rmR$ is $*$-local, under some mild conditions on the functions $\{f_i\}_{1\leq i\leq n}$ and the elements $\{p_i\}_{1\leq i\leq n},$ we show that the resolution is minimal graded with the induced $\Z$-grading. If in addition $\rmR$ is a $*$-local ring and all the $p_i$'s  are nonzero divisors in the maximal ideal, then we can describe combinatorially the $\Z$-graded Betti numbers. Along the way, we disprove Conjecture 3.2.3
in~\cite{mal13b}. 

\subsection{Statement of the main result}\label{main}

From now on, we will work with principal ideals, and fans $\cF$ of cones in $\R^2$ such that $|\cF|$ is also convex, and contained in a half-plane.  In $\R^2$ we have the advantage that there is a standard orientation for both cones and vectors which we use. So assume without loss of generality that $C_1,\ldots, C_\ell,$ the maximal cones in $\cF,$ are sorted clockwise. In addition, we will also assume that for each $i=1,\ldots, \ell,$ the elements in the set $\cH_{Q_i} =\{\bfv_{i1},\ldots, \bfv_{ig_i}\}$  which was defined in \eqref{fine}, are also sorted clockwise.

We will say that the family of functions $\bff$ is {\em strict} with respect to the fan $\cF$ if for each $1\leq i<\ell,$ there is $k_i\in\{1,\ldots, n\}$ such that $f_{k_i}$ is not linear on $C_i\cup C_{i+1}.$ Note that if $\bff$ is not strict with respect to $\cF$ one can take a coarser fan $\cF'$  such that $\bff$ is also a family of piece-wise linear functions compatible with $\cF'$ and strict with respect to this new fan, and moreover from the definition we get
$$\cB_{\bff,\cF,\bfp}=\cB_{\bff,\cF',\bfp},
$$
so we can always assume w.l.o.g. that the set of functions is strict with respect to the fan. 

Set $\cH:=\cup_{i=1}^\ell \cH_{Q_i},$ and denote with $M$ the cardinality of this set. Let $p_1,\ldots, p_n\in\rmR.$ In the sequel, we will set $\bfp^{\bff(\bfv)}=p_1^{f_1(\bfv)}\ldots\, p_n^{f_n(\bfv)},$ for short. For each $\bfv\in\cH,$ let $Y_\bfv$ be a new variable, and set $\bfY=\{Y_\bfv,\,\bfv\in\cH\}.$ 
In light of Theorem \ref{mt1}, we  have an epimorphism
\begin{equation}\label{map}
\begin{array}{ccc}
\rmR[\bfY]&\stackrel{\varphi_0}{\to}& \cB_{\bff,\cF, \bfp}\\
Y_\bfv&\mapsto&\bfp^{\bff(\bfv)}\bfx^{\bfv}.
\end{array}
\end{equation}
Note that $\varphi_0$ is a $\Z^2$-graded map if we declare $\deg(\bfY_\bfv)=\bfv,$ and $\deg(r)=(0,0)$ for all $r\in\rmR.$ If $\rmR$ is $*$-local with maximal ideal $\fm$, then we can see that $\rmR[\bfY]$ is also $*$-local with maximal ideal $(\fm, \bfY)$ with the degree of a monomial given by the sum of its exponents, while the elements of $\rmR$ have degree $0$. We refer to this grading as the induced $\Z$-grading.

The following is the main result of this paper.

\begin{theorem}\label{mresult}
If $|\cF|$ is a proper rational cone contained in $\R^2,$ then there exists $M\in\N$ such that the following is a $\Z^2$-graded resolution of  $\cB_{\bff,\cF,\bfp}$ over $\rmR[\bfY]$:
\begin{equation}\label{eres}
0\to\rmR[\bfY]^{M-2}\stackrel{\varphi_{M-2}}{\to}\ldots\to\rmR[\bfY]^{3{M-1\choose 4}}\stackrel{\varphi_3}{\to}\rmR[\bfY]^{2{M-1\choose 3}}\stackrel{\varphi_2}{\to}\rmR[\bfY]^{M-1\choose 2}\stackrel{\varphi_1}{\to}\rmR[\bfY]\stackrel{\varphi_0}{\to}
\cB_{\bff,\cF,\bfp}\to0 .
\end{equation}
The maps $\varphi_j:\rmR[\bfY]^{j{M-1\choose j+1}}\to\rmR[\bfY]^{(j-1){M-1\choose j}}$ are defined in \eqref{phi1} for $j=1$,  in \eqref{phi2} for $j=2$ and their recursive construction is shown in Section \ref{hs}.
\par 
The resolution is minimal graded over $\rmR[\bfY]$ if $\rmR$ is $*$- local with maximal ideal $\fm$, $\bff$ is strict with respect to $\cF,$ none of the $p_i$'s is a zero divisor in $\rmR,$ and all of them belong to $\fm$. Here, we regard $\rmR[\bfY]$ as graded with the induced $\Z$-grading.
\end{theorem}
Recall from Definition \ref{1.2} that a proper rational cone is neither the whole space nor contained in a line, so in particular we have $M\geq2.$
When  $|\cF|$ is a half-line, then clearly we have that  the map $\varphi_0$ of \eqref{map} is an isomorphism, and hence the resolution stops there. If $|\cF|$ is a line, then 
$\cB_{\bff,\cF,\bfp}$ is isomorphic to $\rmR[Y_1,Y_{-1}]/\langle Y_1Y_{-1}-\bfp^{\bff(1)+\bff(-1)}\rangle,$ so finding its resolution is also not a hard task. 

The situation where $|\cF|=\R^2$ seems to be more tricky, and a resolution such as \eqref{eres} does not hold anymore in general, see the discussion in Section \ref{6}.

Theorem \ref{mresult} follows directly from Theorems \ref{mt}, \ref{machiv}, \ref{machiv2} and \ref{machiv3}. The maps $\varphi_j,\,j=1, 2, \ldots$ are constructed recursively starting from \eqref{phi1},  and \eqref{phi2}. We can actually make explicit the description of all the Betti numbers of this ring when the resolution \eqref{eres} is minimal.
\par

We can state an interesting consequence of our main result to intersection algebras. To that end, let $k$ be a field, $p_1, \ldots, p_n$ indeterminates, $\rmR= k [p_1, \ldots, p_n]$ and $\bfa = (a_1, \ldots, a_n), \bfb=(b_1, \ldots, b_n)$ two integer vector with positive entries. Consider the intersection algebra $\cB(\bfa, \bfb)$ as introduced in Remark~\ref{ff}.

As specified in Remark~\ref{ff}, the $R$-algebra $\cB(\bfa, \bfb)$ is Cohen-Macaulay of dimension $n+2$. To see $\cB(\bfa, \bfb)$ as a fan algebra, consider the nonnegative quadrant in $\Z^2$ and the fan generated by the rays of slope $a_i/b_i, i=1, \ldots, n$. Let us call this fan $\cF$. The fan functions are given by $f_i(r,s) = \max(a_ir, b_is)$, $i=1, \ldots, n$ and these functions are strict with respect to the fan $\cF$. If  $M$ equals the cardinality of $\cH$, then we know that $M$ is the minimal number of generators of $\cB_{\bff,\cF, \bfp} = \cB(\bfa, \bfb)$, as in Theorem~\ref{mresult}.

\begin{corollary}
Using the notation and assumptions above,  $\cB (\bfa, \bfb)$ is Gorenstein if and only if $M=3$. In this case, $a_1=b_1, \ldots, a_n=b_n$ and $$\cB(\bfa, \bfb) = \frac{k[p_1, \ldots, p_n, A, B, C]}{(AB - p_1^{a_1} \cdots p_n^{a_n} C)}.$$

\end{corollary}

\begin{proof}
Note that $\cB(\bfa, \bfb)$ is a quotient of $\rmR[\bfY]$ where $\bfY=\{Y_1, \dots, Y_M\}$. But $\cB(\bfa, \bfb)$ is Cohen-Macaulay, so it is a Cohen-Macaulay $\rmR [\bfY]$-module of finite projective dimension. By the Auslander-Buchsbaum theorem, we have that its projective dimension over $\rmR[\bfY]$ equals $n+M -(n+2)= M-2.$ Our Theorem~\ref{mresult}, together with Proposition 3.2 in~\cite{H}, shows that the type of $\cB(\bfa, \bfb)$ is given the $(M-2)$-th Betti number in our resolution, after localizing at $(\fm, \bfY)$. The formula for the Betti numbers is $\beta_0=1, \beta_i= i \cdot \binom{M-1}{i+1}$, for $i\geq 1$. So, either $M \neq 2$ since the $\bfa, \bfb$ have positive entries, so the type equals $\beta_{M-2}=(M-3+1)\binom{M-1}{M-1} = M-2$.

Finally, $\cB(\bfa, \bfb)$ is Gorenstein if and only if its type is $1$. So, in this case, $M=3$. It is easy to see that this is equivalent to $a_i=b_i$ for all $i=1, \ldots, n$. Moreover, by applying our techniques or by direct computation, one sees that
$$\cB(\bfa, \bfa)  =  \frac{k[p_1, \ldots, p_n, A, B, C]}{(AB - p_1^{a_1} \cdots p_n^{a_n}C)}.$$
\end{proof}

\begin{remark}
{\rm This result generalizes Corollary 2.23 in~\cite{EM14} and Proposition 2.10 in~\cite{ES}.}
\end{remark}

The paper is organized as follows. Section \ref{background} is expository in nature: it describes an algorithmic procedure to construct the Hilbert basis for an integer cone in the plane and outlines the fundamental facts about Gr\"obner bases needed later. Section \ref{m=2} is dedicated to the construction of the presentation ideal of the fan algebra. In Section \ref{varpi1} we study in detail the map $\varphi_1$ of the resolution, while the analysis of the remaining maps and the proof of the main Theorem are completed in Section \ref{hs}.  The paper concludes with some examples which illustrate that the situation for when the fan is the whole plane is more complicated than the results presented here, and hence new ideas are needed in order to tackle this and more general scenarios.

\bigskip
\section{Hilbert and Gr\"obner Bases}\label{background}

We present here a number of results on Hilbert and Gr\"obner bases that are essential in our subsequent sections. This material is generally known to the experts, but since we do not know a good reference for the statements that we need, we include them here for the benefit of the reader.

\subsection{Hilbert bases of two dimensional strongly rational cones} We need a criteria and an algorithm to construct a Hilbert basis of the monoid of integer points lying in the cone generated by $\{(b,a),\,(b',a')\}\subset\R^2.$ Some of the facts below can be found without proof in~\cite{MS}, Example 7.19.

First we start with the following easy lemma, whose proof is left to the reader.

\begin{lemma}
\label{easycase}
Let $\bfv, \bfw$ be two integer vectors such that $\det(\bfv, \bfw) =1$. Then the Hilbert basis of the monoid of integer points lying in the cone generated by $\bfv, \bfw$ is $\{\bfv, \bfw\}$.
\end{lemma}

\begin{proposition}\label{criteria}
Suppose that the sequence $\{\bfv_1,\ldots,\bfv_N\}\subset\Z^2$ lies in an open half-plane (i.e. there exists $\bfv$ such that $\langle \bfv,\bfv_i\rangle>0 \ \forall i=1,\ldots, N),$ and is such that 
\begin{enumerate}
\item[a)]$\det(\bfv_{i+1},\bfv_i)=1,$ for $i=1,\ldots, N-1.$ 
\item[b)] For every $i=1,\ldots, N-2,\, \bfv_{i+1}$ lies in the interior of the parallelogram whose vertices are 
${\bf0},\,\bfv_i,\,\bfv_N,\,\bfv_i+\bfv_N.$
\end{enumerate}
Then
\begin{enumerate}
\item $\det(\bfv_{j},\bfv_i)\geq1$ for all $j>i.$
\item For all $j=1,\ldots, N,$ if we write 
$
\bfv_j=\alpha_j\bfv_1+\beta_j\bfv_N$
with $\alpha_j,\,\beta_j\in\R,$ then \begin{itemize}
\item $\alpha_j,\,\beta_j\geq0$ for all $j=1,\ldots, N,$
\item the sequence $\{\alpha_j\}_{j=1,\ldots, N}$ is strictly decreasing.
\end{itemize}
\item The family $\{\bfv_1,\ldots, \bfv_N\}$ is a Hilbert basis of the monoid generated by the integer points lying in the convex hull of the positive rays generated by $\bfv_1$ and $\bfv_N.$
\item There is a unique Hilbert basis of this monoid.
\end{enumerate}
\end{proposition}
\begin{remark}
Note that condition a) plus the fact that $\{\bfv_1,\ldots, \bfv_N\}$ is contained in a half-plane implies that it is sorted clockwise.
\end{remark}
\begin{proof}
The fact that $\det(\bfv_2,\bfv_1)>0$ implies that $\{\bfv_1,\,\bfv_2\}$ is a basis of the $\R-$vector space $\R^2.$ Write
$$\bfv_3=\lambda_1\,\bfv_1+\lambda_2\,\bfv_2,\,\,\lambda_1,\,\lambda_2\in\R.$$ We then have
$0>\det(\bfv_2,\,\bfv_3)=\lambda_1\det(\bfv_2,\,\bfv_1).$
As $\det(\bfv_2,\,\bfv_1)>0,$ we deduce that $\lambda_1<0.$
This implies that $\lambda_2>0$ as otherwise we would get $\langle \bfv_3, \bfv\rangle<0,$ a contradiction.

We compute then
$\det(\bfv_3,\,\bfv_1)=\lambda_2\det(\bfv_2,\bfv_1)>0,$
and inductively from here one can prove $\det(\bfv_j,\,\bfv_1)>0$ for $j>1.$ More generally, by starting with any index $i\geq1,$ one gets (1).
\par
To prove (2), note that we have -thanks to (1)- $\det(\bfv_N,\bfv_1)>0$, so $\{\bfv_1,\,\bfv_N\}$ is also a basis of $\R^2$. We will prove our claim by induction on $j=1,\ldots, N-1.$
\begin{itemize}
\item For $j=1,$ we have $\alpha_j=1,\,\beta_j=0,$ so the claim follows.
\item For a general $j<N,$ condition (b) implies that there exist $\alpha,\,\beta\in(0,1)$ such that
$$\bfv_{j+1}=\alpha\,\bfv_j+\beta\,\bfv_N.
$$
By using the inductive hypothesis, we have that
$$
\bfv_{j+1}=\alpha\big(\alpha_{j}\bfv_1+\beta_{j}\bfv_N\big)+\beta\,\bfv_N=\alpha\,\alpha_j\bfv_1+(\alpha\,\beta_j+\beta)\bfv_N,
$$
and we deduce that
$0<\alpha_{j+1}=\alpha\,\alpha_j<\alpha_j, \ \ 0<\beta_{j+1}=\alpha\,\beta_j+\beta,$ which completes with the proof of (2).
\par
For (3), denote with $\cM_{\bfv_1,\bfv_N}$ the monoid in $\Z^2$ generated by the integer points lying in the positive rays generated by $\bfv_1,\,\bfv_N.$ Condition (a) implies that every vector in $\cM_{\bfv_1,\bfv_N}$ lies in one cone generated by the positive rays of $\bfv_j,\,\bfv_{j+1}$, for some $j=1,\ldots, N-1.$ As 
$\det(\bfv_{j+1},\bfv_j)=1,$ this implies that any element in this smaller cone with integer coordinates is a linear combination of $\bfv_j,\,\bfv_{j+1}$ with integer (and nonegative) coefficients, by Lemma~\ref{easycase}. This proves that
the family $\{\bfv_1,\ldots, \bfv_N\}$ generates $\cM_{\bfv_1,\bfv_N}.$
\par To show that the generation is minimal, suppose then that for some $j\in\{1,\ldots, N\},$ we have
\begin{equation}\label{uu}
\bfv_j=\sum_{k\neq j}a_k\,\bfv_k,\, a_k\in\N.
\end{equation}
For $i=1,\ldots, N,$ we write $\bfv_i=\alpha_i\bfv_1+\beta_i\bfv_N$, and hence,
$$\bfv_j=\alpha_j\bfv_1+\beta_j\bfv_N=\Big(\sum_{k\neq j}a_k\,\alpha_k\Big)\bfv_1+\Big(\sum_{k\neq j}a_k\,\beta_k\Big)\bfv_N.
$$
As the sequence $\{\alpha_k\}_{k=1,\ldots, N}$ is positive strictly decreasing, we must have $a_k=0$ for $1\leq k<j.$ So,  \eqref{uu} turns into
$$\bfv_j=\sum_{k> j}a_k\,\bfv_k,
$$
which implies then that
$$0=\det(\bfv_j,\,\bfv_j)=\sum_{k>j}a_k\det(\bfv_j,\,\bfv_k)\geq\sum_{k>j}a_j,
$$
i.e. $a_j=0$ for all $j>k,$ which would imply that $\bfv_j$ is the zero vector, a contradiction. This shows that \eqref{uu} is not possible, and hence the family is a minimal set of generators.
\par To prove (4), suppose that $\{\bfv'_1,\ldots,\bfv'_M\}$ is another set of minimal generators of $\cM.$ We have then,
$$
\bfv_i =\sum_{j=1}^Ma_{ij}\,\bfv'_j=\sum_{j=1}^Ma_{ij}\left(\sum_{k=1}^Nb_{jk}\,\bfv_k\right)=\sum_{k=1}^N\left(\sum_{j=1}^Ma_{ij}b_{jk}\right)\bfv_k \ \ i=1,\ldots, N,
$$
with $a_{ij},\,b_{lk}\in\N.$ Note that for all $i=1,\ldots, N,$ we must have $\sum_{k=1}^Ma_{ij}b_{ji}\ge1,$ as otherwise $\{\bfv_1,\ldots, \bfv_N\}$ would not be minimal. We deduce then  that
\begin{equation}\label{prim}
{\bf0}=\left(\sum_{j=1}^Ma_{ij}b_{ji}-1\right)\bfv_i+\sum_{k=1,\,k\neq i}^N\left(\sum_{j=1}^Ma_{ij}b_{jk}\right)\bfv_k, \, i=1,\ldots, N.
\end{equation}
By performing the inner product of the two sides of the equality above with the vector $\bfv$, we deduce that
\begin{equation}\label{sec}
\sum_{j=1}^Ma_{ij}b_{jk}=\delta_{ik}=\left\{\begin{array}{ccc}
1&\,\mbox{if}& i=k\\
0&\,\mbox{if}& i\neq k,
\end{array}\right. \ \ i=1,\ldots, N.
\end{equation}
Set $A=\big(a_{ij}\big)_{1\leq i\leq N,\,1\leq j\leq M},\,B=\big(b_{jk}\big)_{1\leq j\leq M,\,1\leq k\leq N}.$ Identity \eqref{sec} imply that $A\cdot B$ is the identity matrix of $N\times N,$ which implies that $N\leq M.$ A symmetric argument shows that $M\leq N,$ and hence the equality $N=M$ holds. So, we have that $A$ and $B$ are square matrices being inverses of each other.
\par  From here the conclusion is straightforward, as it is easy to see that two square matrices having coefficients in $\N$ being inverse of each other must be permutations of the rows (or columns) of the identity matrix. We deduce then straightforwardly that there is a permutation of the indices of $\{\bfv'_1,\ldots,\bfv'_N\}$ which converts this sequence into $\{\bfv_1,\,\ldots,\,\bfv_N\}.$ This concludes with the proof of the proposition.
\end{itemize}
\end{proof}
\smallskip
The following algorithm would help us compute a Hilbert basis of a given cone. It will be also of use later in the sequel to obtain properties from such bases.
\begin{algorithm}[Computation of a Hilbert basis of a cone]\label{algo} $^{}$
\par \noindent{\bf Input:} Two linearly independent lattice vectors $(b,a),\,(b',a')\in\Z^2.$
\par\noindent{\bf Output:} The Hilbert Basis of the monoid generated by the integer points lying in the strongly rational cone generated by $(b,a)$ and $(b',a').$ 
\par\noindent{\bf Procedure:}

\begin{enumerate}
\item Assume $(b,a)$ and $(b',a')$ have coprime coordinates, otherwise, replace them with $\frac{1}{\gcd(a,b)}(b,a)$ and/or $\frac{1}{\gcd{(a',b')}}(b',a')$ respectively. Assume also w.l.o.g. that $\det\left(\begin{array}{ll} b'& a' \\ b & a \end{array}\right)>0.$
\item Set $\cS_0:=\cS_1:=\{(b',a'),(b,a)\}.$
\item If $\det\big(\cS_1\big)=1,$ return $\cS_0$ and terminate the algorithm. Else
\begin{enumerate}
\item Let $\cS_1=\{(b',a'),\,(b^*,a^*)\}.$Using the Euclidean Algorithm, compute $(u,v)\in\Z^2$  such that 
\begin{equation}\label{oso}\det\left(\begin{array}{ll} u& v \\ b^* & a^* \end{array}\right)=1.\end{equation}
\item Write $(b',a')=\alpha (b^*,a^*)+\beta (u,v),$ with $\alpha,\,\beta\in\Z.$
\item  $\cS_0\cup\{(u,v)+\lceil\frac{\alpha}{\beta}\rceil (b^*,a^*)\}\mapsto \cS_0$ and $\cS_1:=\{(b',a'),(u,v)+\lceil\frac{\alpha}{\beta}\rceil (b^*,a^*)\}.$ Go to step 3.
\end{enumerate}
 \end{enumerate}
 \end{algorithm}

 \begin{proposition}\label{kozu}
Algorithm \ref{algo} is correct.
 \end{proposition}
 
 \begin{proof}
First, it is straightforward to check recursively that $\det(\cS_1)\neq0$ at every step of the algorithm, which implies that every $\beta\in\Z$ computed is not identically zero, so step (c) always has sense. Note also that every time we pass through step 3,  due to \eqref{oso} we have that $\{(u,v),\,(b^*,a^*)\}$ is a $\Z$-basis of $\Z^2$ with positive orientation.  At the beginning of each step, if the algorithm is not going to terminate there, we have that
 $$1<\det(\cS_1)=\det\left(\begin{array}{ll}b'& a' \\ b^* & a^* \end{array}\right)=\beta.$$
 
In the end, we will have $\bfv_N=(b',a'),$ and also at each step, if $\bfv_i=(b^*,a^*)$ then, following the notation of the algorithm,
$$\bfv_{i+1}=(u,v)+\left\lceil\frac{\alpha}{\beta}\right\rceil(b^*,a^*),
$$
with $(b',a')=\alpha(b^*,a^*)+\beta(u,v).$ This implies straightforwardly that
$$\bfv_{i+1}=\left(\big\lceil\frac{\alpha}{\beta}\big\rceil-\frac{\alpha}{\beta}\right)(b^*,a^*)+\frac{1}{\beta}(b',a')=\lambda\,\bfv_{i}+\lambda'\,\bfv_N,
$$
with $0<\lambda,\,\lambda'<1.$ 
 
 This shows that for each $i=1,\dots, N-2,\,\bfv_{i+1}$ lies in the interior of the parallelogram with vertices ${\bf0},\,\bfv_i,\,\bfv_N,\,\bfv_i+\bfv_{N},$ so we can apply Proposition \ref{criteria}.

 At the end of step 3, we now have that - after changing $\cS_1$ with the formula given in (c)- 
 $$\det\left(\begin{array}{ll}
b'&a'\\
u+\lceil\frac{\alpha}{\beta}\rceil b^*& v+\lceil\frac{\alpha}{\beta}\rceil a^*
 \end{array}\right)=\det\left(\begin{array}{ll}
\beta&\alpha\\
1 &\lceil\frac{\alpha}{\beta}\rceil 
 \end{array}
 \right)\cdot\,\det\left(\begin{array}{ll} u&v\\ b^*& a^*\end{array}\right)=\beta\lceil\frac{\alpha}{\beta}\rceil -\alpha<\beta.
 $$
Note that also we have that $\beta\lceil\frac{\alpha}{\beta}\rceil -\alpha$ cannot be equal to zero as otherwise this will imply that $(b',a')$ is $\beta$ has entries that are not relatively prime. So, we now have that the determinat of the ``new'' $\cS_1$ is strictly smaller than $\beta$ and larger than zero. This shows that the algorithm eventually terminates.
\par To conclude, we must show that the output is the Hilbert basis of the set of lattice points lying in the cone generated by $(b,a),\,(b',a').$ Write $\cS_0=\{\bfv_1,\ldots, \bfv_N\}$ sorted clockwise. Note that all these vectors lie in $\N^2.$  By construction, we have already $\det(\bfv_{i+1},\bfv_i)=1,\,i=1,\ldots, N-1.$  Now apply Proposition \ref{criteria}. This concludes the proof.
 \end{proof}
 
 \begin{corollary}\label{corola}
 Suppose  that the cone generated by $\bfv_1$ and $\bfv_N$ is strongly rational, and  
 that the clockwise-sorted family $\{\bfv_1,\ldots, \bfv_N\}\subset\Z^2$ is a Hilbert basis of this monoid.  Then, for $j=1,\ldots, N-2,\, \bfv_{j+1}$  lies in the interior of the parallelogram whose vertices are
 ${\bf0},\,\bfv_j,\,\bfv_N$ and $\bfv_j+\bfv_N.$
 \end{corollary}
 
 \begin{proof}
Apply Algorithm \ref{algo} with input $\bfv_1,\,\bfv_N.$ Thanks to Proposition \ref{kozu}, its ouput gives a Hilbert basis of the aforementioned monoid, which satisfies that every vector in the family -when properly sorted clockwise- lies in the interior of the parallelogram determined by the immediately previous in the list and the last one. This implies the claim for this family. By Proposition \ref{background} (4), the Hilbert basis is unique,  so it must coincide with $\{\bfv_1,\ldots, \bfv_N\}\subset\Z^2.$ This concludes with the proof of the claim.
 \end{proof}
 The following result will be useful in the sequel.
\begin{lemma}\label{alfin}
Let $\{\bfv_1,\ldots, \bfv_N\}\subset\Z^2$ be the clockwise sorted Hilbert basis of the monoid generated by this set. For each $\bfv$ such that $\langle \bfv,\,\bfv_j\rangle>0,\,1\leq j\leq N,$ we have that 
$$\langle \bfv, \bfv_j\rangle \leq \max\{\langle \bfv,\,\bfv_1\rangle,\,\langle \bfv,\,\bfv_N\rangle\}.$$
\end{lemma}

\begin{proof}
Suppose w.l.o.g. that $j=2, $ and write 
$\bfv_2=\lambda\bfv_1+\mu\bfv_N,$ with $0<\lambda,\,\mu<1,$ the strict inequalities are due to the fact that $\bfv_2$ lies in the interior of the parallelogram determined by $\bfv_1$ and $\bfv_N.$  If we show that $\lambda+\mu\leq1,$ then the claim will follow as 
$$\langle \bfv, \bfv_2\rangle=\lambda\langle\bfv, \bfv_1\rangle+\mu\langle\bfv, \bfv_N\rangle \leq  (\lambda+\mu)\max\{\langle \bfv,\,\bfv_1\rangle,\,\langle \bfv,\,\bfv_N\rangle\}$$
To do this, denote with $D$ the determinant of $\{\bfv_N,\,\bfv_1\}.$ According to the Pick's Theorem applied to the parallelogram with vertices $0, \bfv_1, \bfv_N, \bfv_1+\bfv_N$, we get that $D= I +B/2 -1$, where $I$ is the number of interior points, while $B$ is the number of boundary points, and here $B=4$. Note that $D\geq2$ as otherwise there would no be $\bfv_2$ as an interior point. From 
$1=\det(\bfv_2,\bfv_1)=\mu\det(\bfv_N,\bfv_1)=\mu\,D,$ we deduce that $\mu=\frac{1}{D}.$ On the other hand, by solving the linear system $\bfv_2=\lambda\bfv_1+\mu\bfv_N$ using Cramer's rule, we deduce that $\lambda\leq\frac{D-1}{D}$ and hence
$$\lambda+\mu\leq\frac{D-1}{D}+\frac{1}{D}=1,$$
which proves the claim.
\end{proof}

We conclude  with an example of how the algorithm works.

\begin{example}
\label{hb23}
Consider the fan from Example~\ref{ff23}. The positive quadrant is split in two strongly rational cones delimited by the line passing through the origin and $(2,3)$. We will compute Hilbert bases for the monoids of integer points from each of these cones. First consider the cone generated by $\{(0,1), (2,3)\}$. Note that 
$\det \left(
\begin{array}{cc} 
2 & 3 \\
0& 1
\end{array}
\right)
= 2>0$.
Start with the set $S_1=\{(2, 3), (0,1)\}$. We need to find $u, v$ integers such that $ u \cdot 1 - v \cdot 0 =1$, and any $(1, a)$ with $a$ integer satisfies this condition. Write $(2,3)= \alpha (0,1) +\beta (1,a)$ and note that $\lceil \frac{\alpha}{\beta} \rceil = 2-a$, which produces the vector $(1,a) + (2-a)(0,1) = (1,2)$. The new $S_1 = \{ (2,3), (1,2)\}$ and the algorithm ends, since $\det \left(
\begin{array}{cc} 
2 & 3 \\
1& 2
\end{array}
\right)
= 1.$
So this monoid has the Hilbert basis $\{(2,3), (1,2), (0,1)\}$.

Repeating the algorithm for the monoid of integer points in the cone generated by $(1,0), (2,3)$, we get the Hilbert basis  $\{(1,0), (1,1), (2,3)\}.$

\end{example}

\smallskip

\subsection{Gr\"obner bases}
Set $\bfY= Y_1,\ldots, Y_s$ for short. The theory of Gr\"obner bases on a ring  $\rmR[\bfY]$  of polynomials over a Noetherian ring  $\rmR$ and submodules of $\rmR[\bfY]^m$ is a bit more subtle than the one developed for fields, but it is well-known and their properties well stablished. We refer the reader to the presentation given in \cite[Chapter 4]{AL94}, and also for the definitons  and properties of \textit{leading monomial ($\mbox{\rm lm}$),\,leading term ($\mbox{\rm lt}$), leading coefficient ($\mbox{\rm lc}$), least common multiple ($\mbox{\rm lcm}$), reduction modulo a family of polynomials, Division Algorithm, syzygies, Gr\"obner bases, minimal bases} and \textit{S-polynomials}.
\par
 We will assume all along this subsection that the monomial order $\prec$ has been fixed. 

\begin{remark}\label{usare}
For the algorithms described in this text,  when we apply the Division Algorithm to reduce an element of $R[\bfY]$ modulo a family of ordered $\{g_1,\ldots, g_t\},$ in every instance where the division is possible we will choose  the \textit{maximal} index among those available.
\end{remark}

The following result will be useful in the sequel.
\begin{lemma}\label{tecc}
Let $p\in \rmR$ be not a unit in this ring, $g,\,h\in\rmR[\bfY]$ and $G=\{g_1,\ldots, g_t\}\subset\rmR[\bfY]$ with $\mbox{\rm lc}(g_j)=1$ for all $j,$  such that $g$ reduces to $r$ modulo $G$, with no further term of $r$ being able to be reduced modulo $G$. Denote with $\overline{g},\,\overline{G},\,\overline{r}$ the classes of $g,\,G,$ and $r$ modulo $p.$ Then
$\overline{g}$ reduces to $\overline{r}$ modulo  $\overline{G}$
in $\rmR/\langle p\rangle,$ and no term of $\overline{r}$ can be reduced further modulo  $\overline{G}.$
\end{lemma}
\begin{proof}
As the leading coefficients of the $g_i$'s are $1$, then it is clear that no term of $\overline{r}$ can be further reduced modulo $\overline{G}$ as otherwise this reduction could have been done in $\rmR$.
\par To prove the other part of the claim, note that if we look at each step of the division algorithm, we either have $\mbox{\rm lc}(g)\in\langle p\rangle$ or not. In the first case, modulo $p$ this leading coefficient vanishes, but then we will have that $c_{i_0}$ in the algorithm will be an element of $\langle p\rangle$ and hence equal to zero in $\rmR/\langle p\rangle,$ so modulo $p$ we would not have seen any reduction at this step, as the leading coefficient vanishes. In the other case, we have $\overline{c_{i_0}}=\overline{\mbox{\rm lc}(f)}\neq0,$ and hence the reduction happens also modulo $p$. Applying this reasoning recursively, we get the claim.
\end{proof}

Suppose  that $G=\{\bg_1,\ldots,\bg_t\}$ is a Gr\"obner basis of a submodule of $\rmR[\bfY]^m,$ with $\mbox{\rm lc}(\bg_j)=1,\,j=1,\ldots, t.$  Let $\{\tilde{e}_1,\ldots,\tilde{e}_t\}$ be the canonical basis of $\rmR[\bfY]^t,$ and write the S-polynomial
$
S(\bg_i, \bg_j)=\sum_{\nu=1}^tF_{ij\nu}(\bfY)\bg_\nu,
$
with $F_{ij\nu}(\bfY)\in \rmR[\bfY],$ such that
$$
\max_{1\leq\nu\leq t}\{\mbox{\rm lm}\big(F_{ij\nu}(\bfY)\mbox{\rm lm}(\bg_\nu)\big)\}=\mbox{\rm lm}\left(S(\bg_i, \bg_j)\right).
$$
Set $\bX_{i,j}:=\mbox{\rm lcm}\big(\mbox{\rm lm}(\bg_i), \mbox{\rm lm}(\bg_j)\big)$ and define
\begin{equation}\label{bs}
\bfs_{i,j}:=\frac{\bX_{i,j}}{\mbox{\rm lm}(\bg_i)}\,\tilde{e}_i-
\frac{\bX_{i,j}}{\mbox{\rm lm}(\bg_j)}\,\tilde{e}_j-\sum_{\nu=1}^tF_{ij\nu}(\Y)\tilde{e_\nu}\in \rmR[\bfY]^t.
\end{equation}
Let $\cG=\{\bg_1,\ldots,\bg_t\}$ be a sequence  non-zero vectors in $\rmR[\bfY]^m.$  We define an order $\prec_\cG$ on the monomials of $\rmR[\bfY]^t$ as follows
\begin{equation}\label{prorder}
\bfY^\bfalpha\tilde{e}_i\prec_\cF\bfY^\bfbeta\tilde{e}_j\iff
\left\{\begin{array}{llr}
\mbox{\rm lm}(\bfY^\bfalpha\bg_i)\prec\mbox{\rm lm}(\bfY^\bfbeta\bg_j)&\,\mbox{or}&\\
\mbox{\rm lm}(\bfY^\bfalpha\bg_i)=\mbox{\rm lm}(\bfY^\bfbeta\bg_j)&\,\mbox{and} &\,j<i.
\end{array}
\right.
\end{equation}
We call $\prec_\cG$ the {\em order induced} by $\cG$. We have the following theorem which is essentially due to Schreyer, see~\cite{E96}.

\begin{theorem}\label{3713}
If $\cG=\{\bg_1,\ldots,\,\bg_t\}$ is a Gr\"obner basis of a submodule of $\rmR[\bfY]^m,$ with $\mbox{\rm lc}(\bg_j)=1$ for all $j=1,\ldots, t,$ then $\{\bfs_{i,j}\}_{1\leq i<j\leq t}$ is a Gr\"obner basis of the syzygy module $\mbox{\rm syz}(\bg_1,\ldots,\bg_t)\subset \rmR[\bfY]^t$ with respect to $\prec_\cG.$ Moreover,
if $\bfs_{i,j}\neq{\bf0},$ then
\begin{equation}\label{liding}
\mbox{\rm lt}(\bfs_{i,j})=\frac{\bX_{i,j}}{\mbox{\rm lt}(\bg_i)}\,\tilde{e}_i, \quad 1\leq i<j\leq t.
\end{equation}
\end{theorem}
\begin{proof}
The proofs of \cite[Lemma 3.7.9 \& Theorem 3.7.13]{AL94} for the case when $\rmR$ is a field hold straightforwardly to this case.
\end{proof}

\bigskip
\section{$\mbox{ker}(\varphi_0)$}\label{m=2}
In this section, we will compute generators of the kernel of $\varphi_0.$ From now on, unless we give more precisions about the ring $\rmR$, a \textit{minimal set of generators} of an $\rmR$-module will be a set of generators of this module with the property that if we remove one of the elements of the set, it does not generate the same module anymore. Recall from \eqref{map} that we have a presentation
$$\begin{array}{ccc}
\rmR[\bfY]&\stackrel{\varphi_0}{\to}& \cB_{\bff,\cF, \bfp}\\
Y_\bfv&\mapsto&\bfp^{\bff(\bfv)}\bfx^{\bfv},
\end{array}
$$
where $\varphi_0$ is a $\Z^2$-graded map if we declare $\deg(\bfY_\bfv)=\bfv,$ and $\deg(r)=(0,0)$ for all $r\in\rmR.$
Let $\prec$ be the lexicographic monomial order in $\rmR[\bfY]$ which satisfies
$Y_\bfv\prec Y_{\bfv'}$ if and only if $\bfv$ is ``to the left'' of $\bfv'.$  To be more precise, if we sort $\cH = \{\bfv_1, \ldots, \bfv_M \}$ clockwise, then our monomial order is lexicographic order in $\rmR[Y_{\bfv_1},\ldots, Y_{\bfv_M}]$ with $Y_{\bfv_i}\preceq Y_{\bfv_j}$ if and only if $i\leq j.$ If we use this notation, we will denote with $Y_j$ the variable $Y_{\bfv_j}$ to simplify notation.

 We will say that  two distinct vectors $\bfv$ and $\bfv'\in\cH$ are {\em adjacent} if in the open strongly rational cone generated by these two vectors there are no other elements of $\cH.$  Note that being adjacent implies that all the lattice vectors lying  in the cone generated by  $\bfv$ and $\bfv'$ can be expressed as a nonegative integer combination of these two vectors. Moreover, two adjacent vectors lie simultaneously in one and only one of the cones $C_i$, which immediately implies that 
\begin{equation}\label{linear}
f_u(\alpha \bfv+\beta\bfv')=\alpha f_u(\bfv)+\beta f_{u}(\bfv') \ \forall \alpha,\,\beta\in\R_{\geq0},\,u=1,\ldots, n.
\end{equation}

\begin{proposition}\label{curry}
Suppose that $|\cF|$ is a proper rational cone contained in $\R^2.$ If $\bfv$ and $\bfv'$ are distinct, non-adjacent elements of $\cH,$ then there exist  adjacent $\bfw,\,\bfw'\in\cH,\,\alpha,\,\alpha'\in\N$ such that
\begin{equation}\label{S}
\bfS_{\bfv,\bfv'}:=Y_\bfv Y_{\bfv'}-Y_\bfw^\alpha Y_{ \bfw'}^{\alpha'}\bfp^{\bfgamma_{\bfv\bfv'}} \in\ker(\varphi_0),
\end{equation}
with
\begin{equation}\label{nueve}
\bfgamma_{\bfv\bfv'}=\big(f_1(\bfv)+f_1(\bfv')-f_1(\bfv+\bfv'),\ldots,\,f_n(\bfv)+f_n(\bfv')-f_n(\bfv+\bfv')\big).
\end{equation}
The expression defining  $\bfS_{\bfv,\bfv'}$ in \eqref{linear} is unique. Its leading term  with respect to $\prec$ is $Y_{\bfv}Y_{\bfv'}.$
\end{proposition}

Note that the degree of $\bfS_{\bfv,\bfv'}$ with the $\Z^2$-graduation, is equal to $\bfv+\bfv'.$
\begin{proof}
The vector $\bfv+\bfv'$ must belong to one of the strongly rational cones generated by two consecutive elements of $\cH$. Let  $\bfw$ and $\bfw'$ be  two adjacent vectors whose cone contains the former sum. Then, there exist  $\alpha,\,\alpha'\in\N$ with
\begin{equation}\label{key}
\bfv+\bfv'=\alpha\bfw+\alpha'\bfw'.
\end{equation}
Note that we must have  either 
\begin{equation}\label{desig}
\bfv\prec \bfw,\,{\bfw'}\prec {\bfv'} \ \mbox{or} \ {\bfv'}\prec \bfw,\,{\bfw'}\prec {\bfv},
\end{equation} 
due to the way we labeled the vectors in the cone, and moreover the expression \eqref{key} is unique.
By the definition of $\varphi_0$ given in  \eqref{map}, and using \eqref{key}, \eqref{linear} and \eqref{nueve},  it is easy to see that $\bfS_{\bfv,\bfv'}\in\ker(\varphi_0).$ 
The fact that its leading term is $Y_\bfv Y_{\bfv'}$ follows straightforwardly from \eqref{desig}.
 \end{proof}

\begin{definition}
\label{syzygy}
$\bfS_{\bff,\cF,\bfp}:=\{\bfS_{\bfv,\,\bfv'}|\,  \bfv,\, \bfv'\in\cH, \,\bfv\prec\bfv',\, \mbox{non-adjacents}\}.$
\end{definition}
Recall that the cardinality of $\cH$ was denoted by $M$ in the introduction. Therefore, $\#\bfS_{\bff,\cF,\bfp}=\frac{(M-1)(M-2)}{2}={M-1\choose 2}.$
 \begin{lemma}\label{txt}
With the lexicographic monomial order $\prec$ defined above, we have that
\begin{enumerate}
\item $\mbox{\rm Lt}\big(\bfS_{\bff,\cF,\bfp}\big)=\langle Y_\bfv Y_{\bfv'} :\ \bfv,\,\bfv'\in\cH, \,\,\bfv\prec\bfv,'\,  \mbox{non-adjacents}\rangle.
$
\item  $\bfY^\bfc\notin\mbox{\rm Lt}\big(\bfS_{\bff,\cF, \bfp}\big)$ if and only if there exist adjacent vectors $\bfw,\,\bfw'\in\cH$ and $A,\,A'\in\N$ such that
$\bfY^\bfc=Y_\bfw^A\,Y_{\bfw'}^{A'}.$
\end{enumerate}
\end{lemma}

\begin{proof}
The first part of the claim follows straightforwardly from the definition of $\bfS_{\bff,\cF,\bfp}$ and Proposition \ref{curry}.  To prove the rest, it is clear first that any monomial of the form $Y_\bfw^A\,Y_{\bfw'}^{A'},$ with adjacent $\bfw,\,\bfw'\in\cH_{Q_i}$ for some $i,$  and $A,\,A'\in\Z_{\geq0}$ are not divisible by any monomial in 
$\mbox{\rm Lt}\big(\bfS_{\bff,\cF,\bfp}\big),$ which proves one of the implications.
\par To prove the converse, suppose that $\bfY^\bfc=\prod_{\bfv\in\cH}Y_\bfv^{c_\bfv}$ is not divisible by any element in $\mbox{\rm Lt}\big(\bfS_{\bff,\cF, \bfp}\big).$ Due to the first part of the claim, we must have then that $c_\bfw>0$ if and only if $c_{\bfw'}=0$ for any $\bfw'$ non adjacent with $\bfw,$ otherwise, $\bfY^\bfc$ will be divisible by $Y_\bfw Y_{\bfw'}.$ We have the following scenarios:
\begin{itemize}
\item All $c_\bfw=0$, then $\bfY^\bfc=1,$ and the claim holds with $A=A'=0$ in this case.
\item There exists $\bfw_0$ such that $c_{\bfw_0}>0.$ Then, all the $c_{\bfw'}$ with $\bfw'$ non adjacent with $\bfw_0$ will be zero. There are at most two adjacent vectors to $\bfw_0;$ namely, the one immediately ``above''  $\bfw_0,$ and the other which is immediately ``below'' it. Name them $\bfw_0^+$ and $\bfw_0^-$ respectively. 
We may have $c_{\bfw_0}^+>0$ or $c_{\bfw_0^-}>0$, but not the two of them as these two are non-adjacent (they have $\bfw_0$ ``between'' them) and hence the binomial $Y_{\bfw_0^+}Y_{\bfw_0^-}$ belongs to 
$\mbox{\rm Lt}\big(\bfS_{\bff,\bfp,\bfa,\bfb}\big),$ and would divide $Y_{\bfw_0^+}^{c_{\bfw_0^+}}Y_{\bfw0^-}^{c_{\bfw_0^-}}$ if the two exponents are strictly positive. Hence, we have that $\bfY^\bfc=Y_{\bfw_0}^{c_{\bfw_0}}Y_{\bfw_0^+}^{c_{\bfw_0^+}},$ or $\bfY^\bfc=Y_{\bfw_0}^{c_{\bfw_0}}Y_{\bfw_0^-}^{c_{\bfw_0^-}},$ which proves the claim.
\end{itemize}
\end{proof}
Now we are ready for one of the main results of this text.
\begin{theorem}\label{mt}
If $|\cF|$ is a proper rational cone contained in $\R^2,$ no $p_i,\,i=1,\ldots, n,$ is a zero divisor in $\rmR$, then $\bfS_{\bff,\cF, \bfp}$ is a reduced Gr\"obner basis of $\ker(\varphi_0).$ 
\end{theorem}

\begin{proof}
Suppose we have a nonzero polynomial  $P(\bfY)\in\ker(\varphi_0)$ whose leading term is not divisible by any of the leading terms of elements in $\bfS_{\bff,\cF, \bfp}.$ Thanks to Lemma \ref{txt}, we will have then 
\begin{equation}\label{astt}
\mbox{\rm Lt}\big(P\big)=r\,Y_\bfv^AY_{\bfv'}^{A'},
\end{equation}
with $r\in \rmR\setminus\{0\}$  $\bfv,\,\bfv'\in\cH$ adjacent vectors, and $A,A'\in\N.$ If $A=A'=0,$ then $P=r\neq0$ and cannot be in the kernel of $\varphi_0$ as $\varphi_0(r)=r.$  So at least one of the exponents must be different from zero. Suppose also w.l.o.g. that ${\bfv'}\prec {\bfv}.$
\par Write $P= r\,Y_\bfv^AY_{\bfv'}^{A'} + \tilde{P}.$
As $\varphi_0(P(\bfY))=P(\varphi_0(\bfY))=0,$ we must have
\begin{equation}\label{chungo}
r\,\bfp^{Af(\bfv)+A'f(\bfv')}\bfx^{A\bfv+A'\bfv'}=-\tilde{P}|_{\bfY_\bfz\mapsto\bfp^{\bff(\bfz)}\bfx^{\bfz}}=\sum_{\bfw}q_\bfw\bfx^{\bfw}
 \end{equation}
with $q_\bfw\in \rmR,$ and all $\bfw$ in $\cH_Q$ lying ``above'' $\bfv,$
i.e. 
\begin{equation}\label{ozu}
\bfw=\sum_{\bfz\preceq {\bfv}}c_\bfz\,\bfz,\ c_\bfz\in\N, \ \mbox{and} \ c_{\bfv}\leq A.
\end{equation}
This is due to the fact that all the exponents  appearing in the polynomial expansion of $\tilde{P}$ have lexicographic degree strictly smaller than $Y_\bfv^AY_{\bfv'}^{A'},$ i.e. they must be of the form
\begin{equation}\label{monom}
\prod_{\bfz\preceq \bfv}Y_\bfz^{c_\bfz},
\end{equation}
with $c_\bfz\in\N,$ and $c_\bfv\leq A.$ By applying $\varphi_0$ to such a monomial, we get one of the exponents in $\bfx$ appearing in \eqref{ozu}.

As no $p_i$ is a zero divisor in $\rmR$, we then have that  the left hand side of \eqref{chungo} is nonzero. Then, there must be some $\bfx^\bfw$ appearing on the right-hand-side with nonzero coefficient, i.e. there must be a family of nonegative integers  $c_\bfz\in\N,$ with $\bfz\preceq \bfv$ and $c_\bfv\leq A,$
such that
\begin{equation}\label{oti}
A\bfv+A'\bfv'= \sum_{\bfz\preceq\bfv} c_\bfz\,\bfz.
\end{equation}
 Hence, we must have that
 $$\begin{array}{ccl}
 \det\left(A\bfv+A'\bfv',\,\bfv'\right)&=& \det\left(\sum c_\bfz\,\bfz,\,\bfv'\right) \\
 A\,\det(\bfv,\,\bfv')&=&c_\bfv\det(\bfv,\,\bfv')+\sum_{\bfz\prec\bfv} c_\bfz \det(\bfz,\,\bfv'),
 \end{array}
 $$
 and from here we deduce that 
 $$0\leq  (A-c_\bfv)\,\det(\bfv,\,\bfv')=\sum_{\bfz\prec\bfv} c_\bfz \det(\bfz,\,\bfv')\leq0,$$
 which implies that $A=c_\bfv,=c_\bfz=0$ if $\bfz\notin\{\bfv,\,\bfv'\},$ and hence $c_{\bfv'}=A'.$ From \eqref{monom} and \eqref{oti} we deduce then that 
 $\prod_{\bfz\preceq \bfv}Y_\bfz^{c_\bfz}=Y_\bfv^A\,Y_{\bfv'}^{A'},$  which is impossible because the monomial in the left hand side of this equality should be strictly smaller than the one in the right hand side in the lexicographic order.
This contradiction shows that  
$$\mbox{\rm Lt}\big(\bfS_{\bff,\cF,\bfp}\big)=\mbox{\rm Lt}(\ker(\varphi_0)),$$ and hence,  $\bfS_{\bff,\cF, \bfp}$  is a Gr\"obner basis of $\ker(\varphi_0).$ The fact that it is reduced follows straightforwardly from its definition, and the shape of the generators defined in \eqref{S}.
\end{proof}
\bigskip

Even though $\bfS_{\bff,\cF,\bfp}$ is a reduced Gr\"obner basis, it is not necessarily a set of minimal generators of $\ker(\varphi_0),$ see cautionary Example \ref{torito}.

\begin{proposition}\label{corcho}
Suppose that for $i=1,\ldots, n,\, p_i\in \rmR$ is neither a unit nor a zero divisor. Then, the family
$\{\bfp^{\bff(\bfv)} \bfx^{\bfv}; \,\bfv\in\cH\}$ is a minimal set of generators of $\cB_{\bff,\cF, \bfp}$(in the sense that no generator is a product of others) if and only if there are no elements in $\bfS_{\bff,\cF,\bfp}$ of the form $Y_\bfv Y_{\bfv'}-Y_\bfw.$
\end{proposition}

\begin{proof}[Proof of Proposition \ref{corcho}]
Clearly, if there is an element of the form  $Y_\bfv Y_{\bfv'}-Y_{\bfw}$ in  $\bfS_{\bff,\cF,\bfp},$ then the generators are not minimal. Conversely, suppose that $\bfp^{\bff(\bfw_0)} \bfx^{\bfw_0}$ is a product of other elements of the family. This would be equivalent to the existence of an element of the form
\begin{equation}\label{rojo}
\prod_{\bfw}Y_{\bfw}^{\alpha_{\bfw}}-Y_{\bfw_0}
\end{equation} lying in the kernel of $\varphi_0.$ Due to Theorem \ref{mt}, we then have that  if we divide this element by the Gr\"obner basis $\bfS_{\bff,\cF,\bfp}$ the remainder will be zero. As $Y_{\bfw_0}$ is not divisible by any of the (quadratic) leading monomials of the elements of $\bfS_{\bff,\cF, \bfp}$, we must have that the leading term of \eqref{rojo} is actually $\prod_{\bfw}Y_{\bfw}^{\alpha_{\bfw}},$ and hence its total degree in the $\bfY$ variables, namely $\sum_{\bfw} \alpha_{\bfw} $, must be at least equal to 2. Moreover, \eqref{rojo} implies that the remainder of the division of $\prod_{\bfw}Y_{\bfw}^{\alpha_{\bfw}}$ with respect to $\bfS_{\bff,\cF, \bfp}$ is equal to $Y_{\bfw_0}.$
\par In what follows the total degree of a monomial refers to the sum of the exponents appearing in the monomial. Due to the particular shape of the elements of $\bfS_{\bff,\cF, \bfp}$ given in \eqref{S}, every time we divide a monomial of the form  $\prod_{\bfw}Y_{\bfw}^{\alpha_{\bfw}}$ with one of the elements of the Gr\"obner basis, such as $Y_\bfv Y_{\bfv'}-Y_\bfw^\alpha Y_{ \bfw'}^{\alpha'}\bfp^{\bfgamma_{\bfv\bfv'}}$, one of the following two possible outcomes happen:
\begin{enumerate}
\item[i)] the total degree of the remainder does not decrease (this is the case when $\alpha+\alpha'\geq2$  in \eqref{S}),
\item[ii)] the total  degree decreases by one when the division is performed over an element of the form
$Y_\bfv Y_{\bfv'}-Y_\bfz\,\bfp^{\bfgamma_{\bfv\bfv'}}$  (here, $\alpha+\alpha'\leq1$ and $\bfz$ is either $\bfw$ or $\bfw'$).
\end{enumerate}
In any case, the remainder of this division is of the form $\bfp^{\bfgamma}\prod_{\bfw}Y_{\bfw}^{\alpha'_{\bfw}}-Y_{\bfw_0},$ for some $\bfgamma,\,\alpha'_\bfw.$
 At the beginning of the division of $\prod_{\bfw}Y_{\bfw}^{\alpha_{\bfw}}$ with respect to the Gr\"obner basis $\bfS_{\bff,\cF,\bfp}$,  the aforementioned monomial has total degree at least two $2.$ After a finite number of steps within the algorithm, we must obtain the remainder $Y_{\bfw_0},$ which has total degree one. This implies that  step ii) must happen at least once. 
\par 
If there are no elements in $\bfS_{\bff,\cF,\bfp}$ of the form $Y_\bfv Y_{\bfv'}-Y_\bfw,$
then for $Y_\bfv Y_{\bfv'}-Y_\bfz\,\bfp^{\bfgamma_{\bfv\bfv'}}\in\bfS_{\bff,\cF, \bfp},$ we will have that $\bfp^{\bfgamma_{\bfw\bfw'}}\in (p_{i_0})\neq \rmR$ for some $i_0\in\{1,\ldots, n\}.$  
 In $\rmR/(p_{i_0})\neq(0),$ the binomial $Y_\bfv Y_{\bfv'}-Y_\bfz\,\bfp^{\bfgamma_{\bfv\bfv'}}$ actually becomes
$Y_\bfv Y_{\bfv'},$ and the division algorithm stops here, which implies that the remainder in this quotient ring is equal to zero. By Lemma \ref{tecc}, this contradicts the fact that $Y_{\bfw_0}\mod p_{i_0}\neq0.$ 
This proves that an expression like \eqref{rojo} can never exists in $\ker(\varphi_0),$ i.e. the family $\{\bfp^{\bff(\bfv)} \bfx^{\bfv}; \,\bfv\in\cH\}$ is a minimal set of generators of $\cB_{\bff,\cF,\bfp}$.
\end{proof}

Recall from Section \ref{main} that  we say that  $\bff$ is  strict with respect to the fan $\cF$ if for each $1\leq i<\ell,$ there is $k_i\in\{1,\ldots, n\}$ such that $f_{k_i}$ is not linear in $C_i\cup C_{i+1}.$  The fan $\cF$ can always be chosen in such a way that this condition is accomplished.

\begin{definition}
Let $\bfw, \bfw' \in \cH$ and $\bff$ strict with respect to the fan $\cF$. We say that $f_k$ separates $\bfw$ and $\bfw'$ if $f_k(\bfw)+f_k(\bfw')-f_k(\bfw+\bfw')>0$. Similarly, if $f_k$ is not linear in $C_i \cup C_{i+1}$, then we say that $f_k$ separates $C_i$ and $C_{i+1}$.
\end{definition}

We leave it to the reader to check that if $f_k$ separates $C_i$ and $C_{i+1}$, then $f_k$ separates any vectors $\bfw, \bfw'$, where $\bfw \in C_j, j \leq i$ and $\bfw' \in C_l, l \geq {i+1}.$

\begin{proposition}\label{triqui}
If  $\bff$ is strict with respect to $\cF,$ and no $p_i$ is a zero divisor or a unit, then
$\{\bfp^{\bff(\bfv)} \bfx^{\bfv}; \,\bfv\in\cH\}$ is a minimal set of generators of $\cB_{\bff,\cF,\bfp}$.
\end{proposition}

\begin{proof}
Thanks to  Proposition \ref{corcho}, we only need to know that there are no elements of the form  $Y_\bfv Y_{\bfv'}-Y_\bfw$ in $\bfS_{\bff,\cF,\bfp}$. But if this is the case, then we must have $\bfv+\bfv'=\bfw,$ which implies straightforwardly that the vectors $\bfv$ and $\bfv'$ belong to different cones as otherwise the three vectors $\bfv, \bfv', \bfw$ would be part of a Hilbert basis, which would not be minimal because of the relation among them. Hence, the element appearing in $\bfS_{\bff,\bfp,\bfa,\bfb}$ is actually $Y_\bfv Y_{\bfv'}-Y_\bfw\bfp^{\bfgamma_{\bfv\bfv'}}$. The fact that $\bff$ is strict combined with the hypothesis that none of the $p_i$'s is a unit  implies that   $\bfp^{\bfgamma_{\bfv\bfv'}}\neq1.$ This concludes the proof.
\end{proof}

Now we turn to the problem of finding conditions to ensure that $\bfS_{\bff,\cF,\bfp}$ is a minimal set of generators of $\ker(\varphi_0).$
The following claim will be used to prove the main result in this direction.
\begin{proposition}\label{aux}
Suppose that $\bff$ is strict with respect to $\cF,$ and that $|\cF|$ is  a strongly convex rational cone. For $\bfv\in\cH,$ set
$$D_\bfv:=\{(\bfw,\,\bfw')\in\cH^2,\, \bfw\prec \bfw',\, \bfw+\bfw'=\bfv\}.$$
If $D_\bfv\neq\emptyset,$ then  there exists $k\in\{1,\ldots, n\}$ such  $f_k(\bfw)+f_k(\bfw')-f_k(\bfw+\bfw')>0$ for all $(\bfw,\bfw')\in D_\bfv.$
\end{proposition}
\begin{proof}
As before, if we have $\bfw+\bfw'=\bfv,$ as these three vectors belong to $\cH,$ we must have then  that $\bfw$ and $\bfw'$ are in different cones, and $\bfv$ ``between'' them, i.e. $\bfw\prec\bfv\prec\bfw'.$ In particular, we have that $\bfw$ and $\bfw'$ are not adjacents. Suppose that $\bfv\in C_i.$ We have two possible scenarios for $\bfw$ and $\bfw':$
\begin{enumerate}
\item[a)] $\bfw\in C_j,\, j<i$ and $\bfw'\in C_l,\,l\geq i,$
\item[b)] $\bfw\in C_i,$ and $\bfw'\in C_l,\,l> i.$
\end{enumerate}
As $\bff$ is strict with respect to $\cF,$  if we pick $k$ as the index of the function which separates $C_{i-1}$ and $C_i,$ all the cases considered in a) are covered. On the other hand, to deal with the cases appearing in b) we should pick as $k$ the index of the function which separates $C_i$ with $C_{i+1}.$ Note that this index also works for those cases in a) where $\bfw'\in C_j$ with $j>i.$
\par To conclude the proof we must show that we cannot have simultaneously
\begin{equation}\label{asunto}
\bfw+\bfw'=\bfz+\bfz'=\bfv,
\end{equation}
with $\bfw,\,\bfw',\,\bfz,\,\bfz'\in\cH,\ \bfw\in\C_j,\, j<i,\,\bfw'\in C_i,\, \bfz\in C_i, \, \bfz'\in C_{j'},\, j'>i,$ i.e. we only have to consider one  of the two possible $k$'s above  for all the cases. Suppose then that \eqref{asunto} holds, and also w.l.o.g. that the sequence $\{\bfz,\,\bfv,\,\bfz'\}$ is sorted clockwise. We straightforwardly get that
$\{\bfw,\,\bfz,\,\bfv,\,\bfw',\,\bfz'\}$ is sorted clockwise as the three vectors in the middle are in $C_i.$
So we must have that $\bfv$ is an element of the (unique) Hilbert basis of the monoid of lattice points in the cone generated by $\bfz$ and $\bfw'$.  
\par As the support of the fan $|\cF|$ is a strongly rational cone, there exists $\bfv_0\in\Z^2$ such that
$\langle \bfv_0, \bfv'\rangle>0$ for all $\bfv'\in|\cF|.$ From \eqref{asunto} we get then that
$$\max\{\langle \bfv_0,\,\bfw'\rangle, \, \langle \bfv_0,\,\bfz\rangle \}< \langle \bfv_0,\,\bfv\rangle,
$$
which is a contradiction with Lemma \ref{alfin}. This concludes the proof of the Proposition.
\end{proof}
\smallskip

The hypothesis that $|\cF|$ is strongly convex is necessary as the following example shows.
\begin{example}
Consider $\cF$ given by its three maximal cones $C_1, C_2,$ and $C_3$ with
\begin{itemize}
\item $C_1$ being generated by $(-1,0)$ and $(-1,1)$,
\item $C_2$ being generated by $(-1,1)$ and $(1,1),$
\item $C_3$ being generated by $(1,1)$ and $(1,0)$.
\end{itemize}
In this case, we have that $\cH=\{\bfv_1:=(-1,0),\,\bfv_2:=(-1,1),\,\bfv_3:=(0,1),\,\bfv_4:=(1,1),\,\bfv_5:=(1,0)\},$
and these vectors satisfy
$$(-1,0)+(1,1)=(-1,1)+(1,0)=(0,1),
$$
which is a situation like \eqref{asunto}. Now set
$$\begin{array}{ccl}
f_1(\bfv)&=&\left\{\begin{array}{ccl}
0&\mbox{if}& \bfv\in C_1,\\
v_1+v_2&\mbox{if}& \bfv\in C_2\cup C_3
\end{array}
\right.\\ \\
f_2(\bfv)&=&\left\{\begin{array}{ccl}
v_2-v_1&\mbox{if}& \bfv\in C_1\cup C_2,\\
0&\mbox{if}& \bfv\in C_3
\end{array}
\right.\\
\end{array}
$$
As $f_1$ separates $C_1$ from $C_2$ and $f_2$ separates $C_2$ from $C_3,$ we get that $\bff=(f_1,f_2)$ is strict, but
$$
f_1(-1,1)+f_1(1,0)-f_1(0,1)=0=f_2(-1,0)+f_2(1,1)-f_2(0,1).
$$
\end{example}
\begin{remark}\label{ast}
If we do not assume $|\cF|$ being strongly convex but only contained in a hyperplane, the proof of Proposition \ref{aux} can be applied also to this situation, but we would conclude that, given $D_\bfv \neq \emptyset$, there exists a set of indices $S$ with exactly two elements, such that for every $(\bfw, \bfw') \in D_\bfv$, there exists $k \in S$ such that $f_k$ separates $\bfw$ and $\bfw'$; namely, we let $S$ consists of the index  separating $C_{i-1}$ from  $C_i,$ and the index which separates $C_i$ from $C_{i+1}.$
\end{remark}

Recall that if $\bff$ is strict, then for each $i=1,\ldots, \ell-1,$ there is an index $k_i$ such that $f_{k_i}$ is non linear in $C_i\cup C_{i+1}.$ If this is the case, we will assume that we are given already a distinguished list $\{k_1,\ldots, k_{\ell-1}\}$ of such indices.

\begin{theorem}\label{mtb}
With notations as above, if $|\cF|$ is a proper convex rational  cone,  $\bff$  strict with respect to $\cF,$  no  $p_{k_i}$ is  a zero divisor,  $i= 1, \dots, \ell$, and  for $1\leq i\leq j\leq \ell-1,\, \langle p_{k_i}, p_{k_j}\rangle\neq\rmR,$ then $\bfS_{\bff, \cF, \bfp}$  is a set of minimal generators of $\ker(\varphi_0).$
\end{theorem}

\begin{proof}
Note that the condition in the hypothesis also ensure that no $p_{k_i}$ is a unit, $1\leq i\leq \ell-1.$ 
Suppose first that $|\cF|$ is strongly rational convex, and that we have a pair of non-adjacent vectors $\bfv,\,\bfv'\in\cH$ such that
$\bfS_{\bfv,\bfv'}$ is in the ideal generated by all the others elements in $\bfS_{\bff,\cF, \bfp},$ i.e. 
\begin{equation}\label{ttt}
Y_\bfv Y_{\bfv'}-Y_{\bfw}^\alpha Y_{\bfw'}^{\alpha'}\bfp^{\bfgamma_{\bfv\bfv'}}=
\sum_{\bfz,\bfz'} Q_{\bfz,\bfz'}(\bfY)\big(Y_{\bfz}Y_{\bfz'}-Y_{\bfu}^{\alpha_{\bfz\bfz'}}Y_{\bfu'}^{\alpha'_{\bfz\bfz'}}\bfp^{\gamma_{\bfz\bfz'}}\big),
\end{equation}
the sum being over all the pairs $\bfz, \bfz'$ of non-adjacent vectors in $\cH$ except $\bfv, \bfv'.$  If we set all the variables $Y_{\bfz}$ to zero except $Y_{\bfv}$ and $Y_{\bfv'}$  in \eqref{ttt}, we get
\begin{equation}\label{abb}
Y_\bfv Y_{\bfv'}=
-\sum_{\bfz,\bfz'} Q_{\bfz,\bfz'}(0,\ldots, Y_\bfv, Y_{\bfv'}, 0,\ldots, 0)Y_{\bfu}^{\alpha_{\bfz\bfz'}}Y_{\bfu'}^{\alpha'_{\bfz\bfz'}}\bfp^{\gamma_{\bfz\bfz'}},
\end{equation}
where the sum to the right is now indexed over all the pairs of no adjacent $\bfz,\bfz'\in\cH$ such that either
\begin{equation}\label{manolito}
\bfz+\bfz' =\bfv\ \ \mbox{or} \ \ \bfz+\bfz'= \bfv'.
\end{equation} This is due to the fact that no monomial among those of the form $Y_{\bfu}^{\alpha_{\bfz\bfz'}}Y_{\bfu'}^{\alpha'_{\bfz\bfz'}}$ appearing in \eqref{ttt} can be a multiple of $Y_\bfv Y_{\bfv'}$ (as $\bfv$ and $\bfv'$ are no adjacent), but we may have 
\begin{equation}\label{unoo}
\bfz+\bfz'=\alpha_{\bfz\bfz'}\bfu+\alpha'_{\bfz\bfz'}\bfu',
\end{equation}
with $\alpha_{\bfz\bfz'}=1$ (resp. $0$) and $\alpha'_{\bfz\bfz'}=0$ (resp. $1$), and $\bfu\in\{\bfv,\,\bfv'\}$ or $\bfu'\in\{\bfv,\,\bfv'\},$ which implies \eqref{manolito}.

\par By Proposition \ref{aux}, all the $\bfp^{\gamma_{\bfz\bfz'}}$ are multiples of some $p_{k_i}$  (resp. $p_{k_j}$) for those $\bfz+\bfz'=\bfv$
(resp. $\bfz+\bfz'=\bfv'$). So, the right hand side of \eqref{abb} belongs to the ideal $\langle p_{k_i},\,p_{k_j}\rangle\neq\rmR[\bfY].$ By extracting the coefficient of $Y_\bfv Y_{\bfv'}$ on the right hand side of \eqref{abb}, we get then an element in $\langle p_{k_i},\,p_{k_j}\rangle\neq\rmR,$ so the equality in \eqref{abb} cannot hold. This concludes the proof of the Theorem for the case $|\cF|$ being strongly rational.
\par For $|\cF|$ being a closed half-plane, we argue as follows: Recall that we have $\cH=\{\bfv_1,\ldots, \bfv_M\},$ where the sequence is sorted clockwise. So, in this case, we have that $\bfv_1=-\bfv_M.$  
Assume again that we have a situation like \eqref{ttt}. If both $\bfv$ and $\bfv'$ lie in the same cone $C_i,$ then due to Remark \ref{ast}, we will  conclude straightforwardly that 
 all the $\bfp^{\gamma_{\bfz\bfz'}}$  in \eqref{abb} are multiples of either $p_{k_{i-1}}$  or  $p_{k_i}.$ From here, the claim follows straightforwardly.
 \par
 Suppose now that $\bfv\in C_i$ and $\bfv'\in C_{i'}$ with $i<i'.$  If we remove $\bfv_M$ from $\cH$ and apply Proposition \ref{aux} to the subfan generated by $\cH\setminus\{\bfv_M\},$ we would get that all the elements in $D_\bfv$ except possibly the  pair $(\bfv_M, \bfz)$ have an index $k^*_i$ such that $f_{k^*_i}$ separates them.  Due to Remark \ref{ast}, we deduce that we can choose $k^*_i\in\{k_{i-1}, k_i\}.$  
 
 If in addition the pair $(\bfv_M,\bfz)$ belongs to $D_\bfv,$ then we must have  with $\bfv_M+\bfz=\bfv,$ which implies that $\bfz$ is ``to the left'' of $\bfv.$ So,   the exponent $\bfp^{\gamma_{\bfv_M\bfz}}$ will be a multiple of the product of all the $p_{k_j}\, j\geq i.$ 
So, we have that all the  $\bfp^{\gamma_{\bfz\bfz'}}$ coming from \eqref{abb}  with $\bfz+\bfz'=\bfv,$ belong to the ideal $\langle p_{k^*_i},\,\prod_{\ell\geq i} p_{k_\ell}\rangle.$ 

Reasoning symmetrically with $\bfv'$ as before (we remove now $\bfv_1=-\bfv_M$ from $\cH$), we conclude that all the elements in $D_{\bfv'}$ except possibly $(\bfz,\bfv_1)$ induce an element   $\bfp^{\gamma_{\bfz\bfz'}}$  which is multiple of only one between $p_{k_{i'-1}},\,p_{k_{i'}},$ which we denote with $p_{k^*_{i'}},$  and the remaining power  $\bfp^{\gamma_{\bfv_1\bfz}}$ must be a multiple of $\prod_{\ell\leq i'-1}p_{k_\ell}.$ So,  all the  $\bfp^{\gamma_{\bfz\bfz'}}$ coming from \eqref{abb}  with $\bfz+\bfz'=\bfv',$ belong to the ideal $\langle \prod_{\ell\leq i'-1}p_{k_\ell}, p_{k^*_{i'}}\rangle.$ 
\par
As $\langle p_{k^*_i},\,\prod_{\ell\geq i} p_{k_\ell}\rangle\cap\langle \prod_{\ell\leq i'-1}p_{k_\ell}, p_{k^*_{i'}}\rangle\subset\langle p_{k^*_i}, p_{k^*_{i'}}\rangle,$
we can argue as before and conclude with the proof.
\end{proof}

The hypothesis on $\langle p_i,\,p_j\rangle\neq \rmR$ cannot be avoided. See Example \ref{torito}. We illustrate our results with some examples.

\begin{example}\label{torito}
Let $n=2,\, p_1,\,p_2\in \rmR$ not  zero divisors, and $\cF$ being the fan whose maximal cones are the following:

\begin{itemize}
\item $C_1$ being the angular region determined by $(0,1)$ and $(1,3);$
\item $C_2$ being the angular region determined by $(1,3)$ and $(3,1);$
\item $C_3$ being the angular region determined by $(3,1)$ and $(1,0)$.
\end{itemize}
The set $\cH$ for this case  is the following:
$$\cH=\{\bfv_1:=(0,1),\,\bfv_2:=(1,3),\,\bfv_3:=(1,2),\,\bfv_4:=(1,1),\,\bfv_5:=(2,1),\,\bfv_6:=(3, 1), \,\bfv_7:=(1,0)\}.$$
We choose as fan-functions 
$$\begin{array}{lll}
f_1(v_1,v_2)&:=&\max\{3v_1,\,v_2\}\\
f_2(v_1, v_2)&:=&\max\{v_1,\,3v_2\}.
\end{array}
$$
Note that the family $\bff$ is strict.
The map  $\varphi_0$ is described by:
$$\begin{array}{cclccclccclcclc}
Y_{1} &\to&p_1 p_2^3 x_2  & &
Y_{2}&\to&p_1^3 p_2^9 x_1 x_2^3  &&
Y_{3}& \to& p_1^3p_2^6 x_1 x_2^2 &&
Y_{4}&\to&p_1^3p_2^3 x_1 x_2   \\
Y_{5}& \to&p_1^6 p_2^3  x_1^2 x_2 & &
Y_{6}&\to&p_1^9p_2^3 x_1^3 x_2  &&
 Y_{7}&\to&p_1^3 p_2 x_1. & & & & 
\end{array}$$
Theorem \ref{mt} states that $\bfS_{\bff,\cF,\bfp},$ the minimal Gr\"obner basis of this ideal has $15 =\frac{6\times5}{2}$ elements whose leading terms must be the following  $15$ monomials:
$$
\left\{Y_{1} Y_{7} ,\,
Y_{1} Y_{4},\, Y_{7} Y_{3},\,
Y_{1} Y_{3} ,\, 
 Y_{7} Y_{2},\, Y_{4} Y_{2},\,
Y_{7} Y_{4} ,\,  Y_{1} Y_{5},Y_{3} Y_{5}, \, Y_{2} Y_{5},\, 
Y_{7} Y_{5}, \,   Y_{1} Y_{6},  Y_{4} Y_{6},\,   Y_{3} Y_{6}, Y_{2} Y_{6}\right\}.
$$
By computing explicitly this Gr\"obner basis with {\tt Mathematica}, we get the following set:
$$\begin{array}{l}
\left\{Y_{1} Y_{7} - p_1p_2Y_{4},\, 
Y_{1} Y_{4} - p_2 Y_{3},\,  Y_{7} Y_{3}-p_1Y_{4}^2,\,
Y_{1} Y_{3} - p_2 Y_{2},\,  
 Y_{7} Y_{2}-p_1 Y_{4} Y_{3} ,\, Y_{4} Y_{2}-Y_{3}^2, \right.\\ 
Y_{7} Y_{4} - p_1 Y_{5},\, Y_{1} Y_{5}-p_2 Y_{4}^2 ,\, Y_{3} Y_{5}-Y_{4}^3 ,  
 Y_{2} Y_{5} -Y_{4}^2 Y_{3},\, 
Y_{7} Y_{5} - p_1 Y_{6} ,\,   Y_{1} Y_{6}-p_2 Y_{4} Y_{5} ,  \\ \left. 
Y_{4} Y_{6}-Y_{5}^2 ,
\, Y_{3} Y_{6}-Y_{4}^2 Y_{5},  Y_{2} Y_{6}-Y_{4}^4 \right\},
\end{array}
$$
which can easily be seen to be equal to $\bfS_{\bff,\cF, \bfp}.$ If  the ideal $\langle p_1,\,p_2\rangle$ is not the whole ring, then $\bfS_{\bff,\cF, \bfp}$ is also a minimal set of generators of $\ker(\varphi_0),$ and the presentation of $\ker(\varphi_0)$ given in \eqref{map} is minimal. 
\par Suppose now that we have $p_1=p_2+1.$ We then have that $\langle p_1,p_2\rangle=\rmR$, and we easily verify that
$$
Y_{5}Y_{3}-Y_{4}^3=Y_{4}(Y_{7}Y_{3}-p_1Y_{4}^2)-Y_{1,2}(Y_{7}Y_{4}-p_1Y_{5}) -Y_{4}(Y_{1}Y_{5}-p_2Y_{4}^2)+Y_{5}(Y_{1}Y_{4}-p_2Y_{3}),
$$
i.e. the system of generators is not minimal.
\par If we set now $p_2=1$ and impose no further conditions on $p_1$ except that it is not a zero divisor, we can verify straightforwardly that
$$Y_{5}Y_{3}-Y_{4}^3=Y_{4}(Y_{1}Y_{5}-Y_{4}^2)-Y_{5}(Y_{1}Y_{4}-Y_{3}),
$$
i.e. the presentation \eqref{map} is not minimal.
\end{example}
\begin{example}
Let $a,b$ be positive integers, and define $\cF$ as the fan whose maximal cones are
\begin{itemize}
\item $C_1$ being the convex hull of the rays $\R_{\geq0}\cdot (0,1)$ and $\R_{\geq0}\cdot (ab+1,a);$
\item $C_2$ being the convex hull of the rays $\R_{\geq0}\cdot (ab+1,a)$ and $\R_{\geq0}\cdot (1,0).$
\end{itemize}
Set also $p_1=x_1.$ The case $a=1$ and $b=1$ have been studied in \cite{mal13b}. The set $\cH$ in this case, sorted already in the lexicographic order, has the following $a+b+2$ elements:
$$\cH=\{(0,1),(1,1),\,(2,1),\,\ldots,(b,1),\,(ab+1,a),\,((a-1)b+1,a-1),\,\ldots,\,(b,1)(1,0)\}.
$$
Theorem \ref{mt} above states that $\bfS_{\bff,\cF,\bfp}$ has $\frac{(a+b+1)(a+b)}{2}$ elements.
\par
If $a=1,$ we recover Theorem $3.2.1$ in \cite{mal13b}, as it is easy to see  that $\bfS_{\bff,\cF,x_1}$ is the ideal generated by the $2\times2$ minors of  the  $2\times(b+2)$ matrix
$$\left(\begin{array}{llllll}
Y_{(0,1)}&Y_{(1,1)}&\ldots & Y_{(j,1)}&\ldots & Y_{(b+1,1)}\\
Y_{(1,1)}& Y_{(2,1)}&\ldots & Y_{(j+1,1)}& \ldots& Y_{(1,0)}
\end{array}
\right).$$ 
Note that the number of $2\times2$ minors of this matrix is equal to $\frac{(b+2)(b+1)}{2},$ which is the number of elements in the minimal set of generators. However, we do not produce the same set but it is very easy to see that the ideal generated is the same.

\par The other extremal case of this situation -also treated in \cite{mal13b}- was when  $b=1.$ Thanks to Theorem \ref{mt} we can show that Conjecture $3.2.3$ in \cite{mal13b} does not hold.
\par Indeed, we have in this case $a_1=a,\,b_1=a+1.$ So, the cones here are the following:
$$\begin{array}{lll}
C_1&=&\langle (0,1), (a+1,a)\rangle\\
C_2&=&\langle (a+1,a),\, (1,0)\rangle;
\end{array}
$$
and then we have
$$\begin{array}{lll}
\cH_{Q_0}&=&\{ (0,1), (1,1),\, (a+1,a)\}\\
\cH_{Q_1}&=&\{ (a+1,a),\, (a, a-1),\,\ldots, (2,1),\,(1,0)\}.
\end{array}
$$
The cardinality of $\cH$ is then $a+3.$ Theorem \ref{mt} gives then a list of (minimal) generators of $\cB_{\bff, \cF, x_1}.$ Denote $a=n \geq 2$ to faciliate the comparison to Conjecture $3.2.3$ in \cite{mal13b}. One of them is $Y_{(0,1)}Y_{(2,1)}-Y_{(1,2)}^2x_1^{n-1}.$ We claim that this element does not belong to the ideal proposed in \cite[Conjecture 3.2.3]{mal13b}. Indeed, with the notation there, this element translates as $x_1x_4-x_3^2x_{n+4}^{n-1},$ which can be easily seen to be part of the kernel of $\varphi_0.$ But if we try to write it as a polynomial combination of the prospective generators of the ideal, and then set all of the variables equal to zero except $x_1$ and $x_4,$ we would have
$$x_1x_4=p(x_1,x_4)\cdot x_1x_4^n+q(x_1,x_4)\cdot x_4^{2n},
$$
which is impossible.
 \end{example}

\begin{example}
\label{phi0}
Let us produce a presentation for $\cB_{\rmR, p}(3,2)= \sum_{r,s} (p^{\max(3r,2s)})x_1^rx_2^s \subset \rmR[x_1,x_2]$, where  $p \in \rmR$ is a nonzerodivisor, see Remark~\ref{ff}, (4). As in Example~\ref{hb23}, $\cH = \{\bfv_1:=(0,1),\,\bfv_2:=(1,2),\, \bfv_3:=(2,3),\, \bfv_4:=(1,1),\,\bfv_5:=(1,0) \}$ and $M=5$. So, $\cB_{\rmR, p}(3,2) = \rmR[p^3x_1, p^2x_2, p^4x_1x_2^2, p^3x_1x_2, p^6x_1^2x_2^3]$. The map $$\varphi_0: \rmR[\bfY] \to \cB$$ sends 
$$Y_1 \to p^2x_2, \ Y_2 \to p^4x_1x_2^2,  \ Y_3 \to p^6 x_1^2x_2^3,$$
$$Y_4 \to p^3x_1x_2, \ Y_5 \to p^3x_1.$$
Recall that our lexicographic order is of the form $Y_1\prec Y_2\prec Y_3\prec Y_4\prec Y_5$. We have $\binom{5-1}{2} = 6$ relations in the kernel of this map.  They are $$S_{1,3}=Y_1Y_3-Y_2^2, \ S_{1,4}=Y_1Y_4-pY_2, \ S_{1,5}=Y_1Y_5-p^2Y_4$$
$$S_{2,4}=Y_2Y_4-pY_3, \ S_{2,5}=Y_2Y_5-pY^2_4, \ S_{3,5}=Y_3Y_5-Y_4^3.$$ In this example, we have $\bfv_1+\bfv_5=\bfv_4$ and
$$\max(3,0) +\max(0,2) -\max(3,2) = 5-3=2,$$ which gives the relation $Y_1Y_5-p^2Y_4$. We denote this relation $S_{1,5}$ instead of $S_{\bfv_1, \bfv_5}$. The rest are produced similarly. Therefore, $\cB_{\rmR, p}(3,2) = \rmR[Y_1, \ldots, Y_5]/I$, where $$I=\langle Y_1Y_3-Y_2^2, Y_1Y_4-pY_2, Y_1Y_5-p^2Y_4, Y_2Y_4-pY_3, Y_2Y_5-pY^2_4, Y_3Y_5-Y_4^3\rangle.$$

\end{example}

\bigskip
\section{The map $\varphi_1$}\label{varpi1}
With the notation and terminology of Section \ref{background},  we are going to study the syzygy module $\mbox{\rm syz}\big(\bfS_{\bff,\cF,\bfp}\big)\subset \rmR[\bfY]^{M-1\choose 2},$ where  $R[\bfY]^{M-1\choose 2}$ is the free module of rank ${M-1\choose 2}$ over $\rmR[\bfY]$ with basis $\{e_{(\bfv,\,\bfv')},\,\bfv\prec\bfv'\in\cH,\,\mbox{non-adjacents}\}.$ So, we have the map
\begin{equation}\label{phi1}
\begin{array}{ccccc}
\mbox{\rm syz}\big(\bfS_{\bff,\cF, \bfp}\big)& \subset &\rmR[\bfY]^{M-1\choose 2}&\stackrel{\varphi_1}{\to}& \rmR[\bfY]\\
& & e_{(\bfv,\,\bfv')}&\mapsto & \bfS_{\bfv,\,\bfv'}.
\end{array}
\end{equation}
This syzygy submodule is also $\N^2$-homogeneous if we declare
\begin{equation}\label{tocho}
\deg(e_{(\bfv,\,\bfv')}):=\bfv+\bfv'.
\end{equation}

For $\bfv,\bfv',\,\bfw,\,\bfw'\in\cH$ such that
neither $\bfv,\,\bfv'$ nor $\bfw,\,\bfw'$ are adjacents, we denote with
$\bfs_{\bfv,\bfv',\bfw,\bfw'}$ the $S$-polynomial between $\bfS_{\bfv,\bfv'}$ and $\bfS_{\bfw,\bfw'}.$
As $\bfs_{\bfv,\bfv',\bfw,\bfw'}\in\ker(\varphi_0),$ the division algorithm of this polynomial against $\bfS_{\bff,\cF, \bfp}$ gives
\begin{equation}\label{preraw}
\bfs_{\bfv,\bfv',\bfw,\bfw'}=\sum_{\bfz,\bfz'}q_{\bfz,\bfz'}(\bfY)\bfS_{\bfz,\bfz'},
\end{equation}
with $q_{\bfz,\bfz'}(\bfY)\in\rmR[\bfY],$ the sum being over all pair of non-adjacent $\bfz\prec\bfz'\in\cH.$
Let $\bfY^\bfalpha,\,\bfY^{\bfalpha'}$ be such  that
$\bfY^\bfalpha\bfS_{\bfv,\bfv'}-\bfY^{\bfalpha'}\bfS_{\bfw,\bfw'}=\bfs_{\bfv,\bfv',\bfw,\bfw'}.$ We then easily have that
\begin{equation}\label{oto}
\cS_{\bfv,\bfv',\bfw,\bfw'}:=\bfY^\bfalpha e_{(\bfv,\bfv')}-\bfY^{\bfalpha'}e_{(\bfw,\bfw')}-\sum_{\bfz,\bfz'}q_{\bfz,\bfz'}(\bfY)e_{(\bfz,\bfz')} \in \mbox{\rm syz}\big(\bfS_{\bff,\cF, \bfp}\big).
\end{equation}
Note that \begin{equation}\label{otor}\bfY^\bfalpha\in\{Y_\bfw,\, Y_{\bfw'},\,Y_{\bfw}Y_{\bfw'}\},\end{equation} and similarly  $\bfY^{\bfalpha'}\in\{Y_\bfv,\, Y_{\bfv'},\,Y_{\bfv}Y_{\bfv'}\}.$ So, we have that
$$\deg\left(\cS_{\bfv,\bfv',\bfw,\bfw'}\right)=\left\{\begin{array}{ccc}
\bfv+\bfv'+\bfw+\bfw'& \mbox{if} & \{\bfv,\,\bfv'\}\cap\{\bfw,\,\bfw'\}=\emptyset\\
\bfv+\bfv'+\bfw& \mbox{if} & \bfw'\in\{\bfv,\,\bfv'\}\\
\bfv+\bfv'+\bfw'& \mbox{if} & \bfw\in\{\bfv,\,\bfv'\}.
\end{array}
\right.
$$
The following result will be useful in the sequel. It follows straightforwardly by noticing the structure of the polynomials $\bfS_{\bfv,\,\bfv'}$ defined in \eqref{S}.
\begin{lemma}\label{tool}
Write $\cS_{\bfv,\bfv',\bfw,\bfw'}$ as in \eqref{oto}. Then, all the vectors $\bfz,\,\bfz'$ appearing in the expansion of this syzygy belong to the cone generated by $\{\bfv,\bfv',\bfw,\bfw'\}.$ If $\bfv=\bfw$ and $\bfv\prec\bfv',\,\bfw\prec\bfw',$ then the pairs $\bfz,\,\bfz'$ in \eqref{oto} actually belong to the cone generated by $\{\bfz,\bfv',\bfw'\},$ with $\bfv\prec\bfz$ and $\bfv,\,\bfz$ adjacent.
\end{lemma}
 We will apply the results of Section \ref{background} to this situation. To do this, we need to sort the elements of the set $\bfS_{\bff,\cF,\bfp}.$  As they are indexed by  pairs  of non-adjacent  $\{\bfv_i,\bfv_j\}$ with $1\leq i<j+1\leq N,$ we will sort them by using the standard lexicographic order in $\N^2,$ i.e. $(i,j)\prec (i',j')$ if and only if $i<i',$ or $i=i'$ and $j<j'.$

 Theorem \ref{3713} implies then that the collection $\{\cS_{\bfv_i,\bfv_j,\bfv_{i'},\bfv_{j'}}\}$ over all pairs  of non-adjacent  $(i,j)\prec(i',j'),$ generates $\mbox{\rm syz}\big(\bfS_{\bff,\cF, \bfp}\big)$ as an $\rmR[\bfY]$-module, and it is a Gr\"obner basis of this submodule for the induced order by this set, which will be denoted by $\prec_\bfS.$  Note that this induced order depends on the way we have sorted the elements of $\bfS_{\bff,\cF,\bfp}$ (the lexicographic order above) but it is not any kind of lexicographic order in  $\rmR[\bfY]^{M-1\choose 2}.$ 
 To simplify notation, we will also denote with $e_{(i, j)}$ the element $e_{(\bfv_i,\,\bfv_j)}.$ Recall that we will also denote with $Y_j$ the variable $Y_{\bfv_j}.$

\begin{definition}
For $i,\,j\in\{1,\ldots, M\}$ with $i+1<j,$ and $k\in\{i+1,i+2,\ldots, M\}\setminus\{j-1,j\},$ we define \begin{equation}\label{sisi}
\cs_{i,j,k}=\left\{\begin{array}{cl}
{\cS}_{\bfv_i,\,\bfv_j,\,\bfv_i,\,\bfv_k}& \mbox{if} \ i<j<k \\
{\cS}_{\bfv_i,\,\bfv_j,\,\bfv_k,\,\bfv_j}& \mbox{if} \ i<k<j.
\end{array}
\right.
\end{equation}
\end{definition}

A triplet $(i,j,k)$ for which $\cs_{i,j,k}$ is defined will be called ``admissible''. For such an admissible triplet, by computing explicitly $\cs_{i,j,k},$ we get that it is equal to either
\begin{equation}\label{mamn}
Y_{{k}}e_{(i,\,j)}-Y_{{j}}e_{(i,\,k)}+\sum_{u<u'-1}q_{i,j,k,u,u'}(\bfY)e_{(u,\,u')},\ \mbox{or} \ \
Y_{{k}}e_{(i,\,j)}-Y_{{i}}e_{(k,\,j)}+\sum_{u<u'-1}q_{i,j,k,u,u'}(\bfY)e_{(u,\,u')},
\end{equation}
with 
\begin{equation}\label{osit}
i\leq u<u'\leq\max\{j,k\}.
\end{equation}

\begin{lemma}\label{au1}
With notation as above, if a  term of the form $rY_{t}$ with $r\in\rmR\setminus\{0\}$ appears in $q_{i,j,k,u,u'}(\bfY),$ then the triplet $(u,u',t)$ is not admissible.
\end{lemma}
\begin{proof}
The presence of such a term implies that, in the process of doing the division algorithm \eqref{preraw}, we would have ran into a nonzero term of the form $rY_{t}Y_{u}Y_{{u'}}.$ From here, to produce the element $rY_{t}e_{(u,u')},$ we should made the division of this term with the binomial $S_{\bfv_u,\bfv_{u'}}$ whose leading term is $Y_{u}Y_{{u'}}.$ But if the triplet $(u,u',t)$ is admissible, we would have that $u<u',\,u<t,$ and $t\in\{u+1,u+2,\ldots, M\}\setminus\{u'-1,u'\}.$  Due to the way we have indexed the set $\bfS_{\bff,\cF,\bfp}$ above, and the fact that our Division Algorithm defined in Remark \ref{usare} chooses the maximal index among those dividing the leading term, we have that  the division of this term must be done by $S_{\bfv_{u'},\bfv_t},$ except in the case when $t=u'+1$ (as $t=u'-1$ cannot happen due to our choice of $t$), in which case the division should be done by $S_{\bfv_u,\bfv_t}.$ In none of the cases we get a term with coordinate  $e_{(u,u')}.$ This concludes the proof of the claim.
\end{proof}

\begin{lemma}\label{au2}
For an admissible triplet $(i,j,k),$ if  $\bfv_i,\,\bfv_j$ and $\bfv_k$ are all contained in one single cone of the fan, then $\cs_{i,j,k}\in \langle Y_{1},\ldots, Y_{M}\rangle\mbox{\rm syz}\big(\bfS_{\bff,\cF, \bfp}\big).$
\end{lemma}
\begin{proof}
Suppose w.l.o.g. $i<j<k,$ the other case being treated analogously. We have
$$
\bfS_{\bfv_{i},\bfv_{j}}=Y_{{i}} Y_{{j}}-Y_{t}^{\alpha} Y_{ {t+1}}^{\alpha'}
\ \ \mbox{and} \ \
\bfS_{\bfv_{i},\bfv_{k}}=Y_{{i}} Y_{{k}}-Y_{{u}}^{\beta} Y_{ {u+1}}^{\beta'},
$$
with $i\leq t<j,\,i\leq u<k,\,\alpha+\alpha'\geq2$ and $\beta+\beta'\geq2,$ as all these identities involve vectors of the Hilbert basis of the semigroup which generates the common cone in $\R^2.$ When computing the $S$-polynomial between these two binomials, we will end up with another binomial each of its terms has now total degree at least $3$. Thanks to Lemma \ref{tool}, we know that the whole division algorithm which will produce the expression \eqref{oto} will involve only pairs of non-adjacent vectors $\bfv_t,\,\bfv_{t'}$ in the common cone. This implies that at each step of the Division Algorithm, a monomial of total degree $\kappa$ is replaced by another of degree {\em at least} $\kappa$ as the division is always done against expressions of the form $Y_{t} Y_{{t'}}-Y_{m}^\gamma Y_{ {m+1}}^{\gamma'}$ with $\gamma+\gamma'\geq2.$
We conclude then that the polynomial $q_{i,j,k,u,u'}(\bfY)$ from \eqref{mamn} lies in 
$\langle Y_{1},\ldots, Y_{M}\rangle\mbox{\rm syz}\big(\bfS_{\bff,\cF, \bfp}\big).$
\end{proof}

We recall again that for $\bff$ is strict,  we can choose a set of indices $\{k_1,\ldots, k_{\ell-1}\}$ such that $f_{k_i}$ is non linear in $C_i\cup C_{i+1},\,i=1,\ldots, \ell-1.$ 

\begin{lemma}\label{yyavan}
With notation as above, if a nonzero constant term of the form $r\in\rmR$  appears in the support of one of the polynomials $q_{i,j,k,u,u'}(\bfY)$ from \eqref{mamn}, and $\bff$ is strict, then there is an index $s\in\{2,\ldots, \ell-1\}$ depending on $u$ such that $r\in\langle p_{k_{s-1}}, p_{k_s}\rangle.$
\end{lemma}
\begin{proof}
As before, we can suppose w.l.o.g. $i<j<k,$ the other case being treated analogously. 
Choose $s$ as the index such that $\bfv_u\in C_{s}.$ As $\bff$ is strict, each of the $(u,u')$ coming from \eqref{mamn} which produce a nontrivial $q_{i,j,k,u,u'}(\bfY)$ must satisfy \eqref{osit}. Moreover, due to Lemma \ref{au2}, we cannot have the three vectors $\bfv_i,\,\bfv_j,\,\bfv_k$ lying in the same cone.  This implies that there are two possible scenarios:
\begin{enumerate}
\item $\bfv_{i}\in \cH_{Q_{s^*}},$ with $s^*< s,$
\item $\bfv_{i}\in \cH_{Q_{s}},$ and $\bfv_{k}\in\cH_{Q_{s^*}},$ with $s<s^*.$
\end{enumerate}
In the first case, by unravelling the Division Algorithm to produce $\cs_{i,j,k},$ one concludes that $r\in \langle p_{k_{s-1}}\rangle.$ In the second case, via the same analysis and using also Lemma \ref{au2} we conclude that 
$r\in \langle p_{k_s}\rangle.$ This concludes with the proof of the Lemma.
\end{proof}

Denote with $\bfS^{(2)}_{\bff,\cF, \bfp}:=\{\cs_{i,j,k}\},$ where the indices $(i,j,k)$ run over all $1\leq i<j-1\leq M,\,k\in\{i+1,\ldots, n\}\setminus\{j-1,j\}.$
\begin{lemma}\label{pre1}
For $u, N\in\N$ such that $1\leq u< N,$ we have  
$\sum_{j=1}^{N-u}j{j+u\choose u}=(u+1){N+1\choose u+2}.$
\end{lemma}
\begin{proof}
By global induction on $N$. If $N=2,$ the only possible value is $u=1,$ and the claim follows straightforwardly. In the general case, if $u<N,$ then by using the induction hypothesis we compute
$$\begin{array}{ccl}
\sum_{j=1}^{N+1-u}j{j+u\choose u}&=&(N+1-u){N+1\choose u}+\sum_{j=1}^{N-u}j{j+u\choose u}\\
&=&(N+1-u){N+1\choose u}+(u+1){N+1\choose u+2}=\frac{(N+1)!}{u!(N-u-1)!}\left(\frac{1}{N-u}+\frac{1}{u+2}\right)\\
&=&(u+1){N+2\choose u+2}.
\end{array}
$$
The case $u=N<N+1$ follows by a direct computation.
\end{proof}

The following set will be useful in computing syzygies.

\begin{definition}
\label{sets}
For $1\leq u\leq M-3,$ let $S_{u}$ be the set of those ordered $(u+2)$-tuples $(i,j,k_1,\ldots, k_u)$ such that $i<j-1\leq M-1,\,1\leq i<k_1<k_2<\ldots<k_u\leq M$ and $k_m\notin\{j-1,j\}$ for all $m=1,\ldots, u.$
\end{definition}

\begin{lemma}\label{comb}
The cardinality of $S_u$ is equal to $(u+1){M-1\choose u+2}.$

\end{lemma}
\begin{proof}
To count the elements of $S_{u}$ we proceed as follows: as the indices $j-1$ and $j$ must be disjoint from the whole sequence $i<k_1<\ldots<k_u,$ we first count the number of  sequences $k_1<k_2<\ldots <k_u$  contained in the set $\{i+1,\ldots, M-1\}.$ This number is clearly ${M-1-i\choose u}.$ For each such distribution, we can ``choose'' as the value of $j-1$ any of the remaining numbers not taken yet, and ``expand'' the sequence by adding $j$ to the right (and hence expanding the cardinality of the set up to $M$ .  The number of positions available for $j-1$ is clearly  equal to  $(M-1-u-i).$ So, the cardinality of $S_u$ is
$$\sum_{i=1}^{M-u-2}{M-1-i\choose u}(M-1-u-i)=\sum_{j=1}^{M-u-2}j{j+u\choose u}=(u+1){M-1\choose u+2},
$$
the last equality follows thanks to Lemma \ref{pre1}
\end{proof}
\begin{corollary}\label{coxis}
The cardinality of $\bfS^{(2)}_{\bff,\cF, \bfp}$ is equal to $2{M-1\choose 3}.$
\end{corollary}
\begin{proof}
The claim follows straightforwardly from Lemma \ref{comb} with $u=1.$
\end{proof}

\begin{theorem}\label{machiv}
If  $|\cF|$ is a proper convex polyhedral cone,  $\bfS^{(2)}_{\bff,\cF, \bfp}$ is a minimal Gr\"obner basis of $\mbox{\rm syz}\big(\bfS_{\bff,\cF, \bfp}\big)$ for the induced order $\prec_\cS.$ If in addition,
 $\bff$ is strict with respect to $\cF,$ none of the  $p_{k_i}$'s is  a zero divisor,  and  for $1\leq i\leq j\leq \ell-1,\,\langle p_{k_i}, p_{k_j}\rangle\neq\rmR,$ then it is also a minimal set of generators of this submodule.
\end{theorem}

\begin{proof}
From \eqref{sisi} and \eqref{liding}, we deduce straightforwardly that for the monomial order $\prec_\cS,$
\begin{equation}\label{juan}
\mbox{\rm lt}(\cs_{i,j,k})=Y_{k}\,e_{(i,\,j)}.
\end{equation}
Theorem \ref{3713} implies that the set 
$\{\cS_{\bfv_i,\,\bfv_j,\,\bfv_k,\,\bfv_u}\}$ indexed by pairs of pairs of  non-adjacent  $(i,j)\prec (k,u)$ 
in $\cH$, generate $\mbox{\rm syz}\big(\bfS_{\bff,\cF, \bfp}\big).$ In addition, from \eqref{otor} we get that 
the leading term of each of these elements is a monomial of total degree one or two in the variables  $\bfY$'s times $e_{(i,\,j)}.$ To show that it is a Gr\"obner basis, it will be enough to show that any of these monomials is a multiple of $Y_{k},$ with $k\in\{i+1,\ldots, M\}\setminus\{j-1,j\}.$ 
\par We straightforwardly verify that, for $(i,j)\prec (k,u),$ 
\begin{equation}\label{oxis}
\frac{\mbox{lcm}(Y_{i}Y_{j},Y_{k}Y_{u})}{Y_{i}Y_{j}}=\mbox{lt}\big(\cS_{\bfv_i,\,\bfv_j,\,\bfv_k,\,\bfv_u}\big)\in\{Y_{k},\,Y_{u},\,Y_{k}Y_{u}\}.
\end{equation}
It is also easy to conclude that neither $Y_{i}$ nor $Y_{j}$ can appear in \eqref{oxis}.  So, the piece in degree $(e_i, e_j)$ of $\mbox{Lt}\Big(\mbox{\rm syz}\big(\bfS_{\bff,\cF, \bfp}\big)\Big)$ is contained in
$$\langle Y_{i+1}, \ldots, Y_{j-1}, Y_{j+1},\ldots, Y_M\rangle \, e_{(i,j)}.
$$
From \eqref{juan}, we only need to show that $Y_{{j-1}}\,e_{(i,\,j)}\notin\mbox{\rm Lt}\Big(\mbox{\rm syz}\big(\bfS_{\bff,\cF, \bfp}\big)\Big).$
But if this was the case, then it should come from either $\cS_{\bfv_i,\,\bfv_j,\,\bfv_i,\,\bfv_{j-1}}$ - which cannot happen as in our order we have $(i,j-1)\prec(i,j)$, or from $\cS_{\bfv_i, \bfv_j, \bfv_{j-1},\bfv_j}.$ But $\bfv_{j-1}$ and $\bfv_j$ are adjacents, so the latter expression has no sense.  These arguments imply that $\bfS^{(2)}_{\bff,\cF, \bfp}$ is a Gr\"obner basis of $\mbox{\rm syz}\big(\bfS_{\bff,\cF, \bfp}\big).$ The fact that it is minimal follows straightforwardly as if we remove one of these elements, the leading module generated by the remaining family will not contain the leading term of the removed syzygy.
\par
To prove now that it is also a minimal set of generators of this submodule given the extra hypothesis, we proceed as follows: assume that one of them is a polynomial combination of the others. So, we must have
\begin{equation}\label{sweet}
\cs_{i,j,k}=\sum_{i',j',k'}q_{i',j',k'}(\bfY)\cs_{i',j',k'}
\end{equation}
for admissible triplets $(i',j',k'),$ with $q_{i',j',k'}(\bfY)\in \rmR[\bfY]$ of $\Z^2$-degree $\bfv_i+\bfv_j+\bfv_k-\bfv_{i'}-\bfv_{j'}-\bfv_{k'}.$  Thanks to Lemma \ref{tool}, we get that the $e_{(i,j)}$ coordinates in \eqref{sweet} produces
\begin{equation}\label{avobe}
Y_{k}= -q_{i,k,j}(\bfY)Y_{k}+\sum_{i'_1,j'_1,k'_1}q_{i'_1,j'_1,k'_1}(\bfY)q_{i'_1,j'_1,k'_1,i,j}(\bfY) \end{equation}
for some triplets $(i'_1,\,j'_1,\,k'_1)$ which must verify  $ i'_1\leq i< j \leq \max\{j'_1,k'_1\}$ due to \eqref{osit}.  Note that $-q_{i,k,j}(\bfY)Y_{k}$ appears if and only if $(i,k,j)$ is also an admissible triplet.
\par
Now we evaluate $Y_{u}\mapsto0$ for all $u\neq k$ in \eqref{avobe} to get
\begin{equation}\label{29}Y_{k}= -q_{i,k,j}(0,\ldots, Y_{k},\ldots, 0)Y_{k}+\sum_{i'_1,j'_1,k'_1}q_{i'_1,j'_1,k'_1}(0,\ldots, Y_{k}, \ldots, 0)q_{i'_1,j'_1,k'_1,i,j}(0,\ldots, Y_{k},\ldots, 0)
\end{equation}
which we write
\begin{equation}\label{avobbe}
Y_{k}= -Q_{i,k,j}(Y_{k})Y_{k}+\sum_{i'_1,j'_1,k'_1}Q_{i'_1,j'_1,k'_1}(Y_{k})Q_{i'_1,j'_1,k'_1,i,j}(Y_{k})\in\rmR[Y_{k}].
\end{equation}
As the polynomials $q_{i,k,j}(\bfY),\,q_{i'_1,j'_1,k'_1}(\bfY),\,q_{i'_1,j'_1,k'_1,i,j}(\bfY)\in\rmR[\bfY]$ are $\Z^2$-graded, this applies also to $Q_{i,k,j}(Y_{k}),\,Q_{i'_1,j'_1,k'_1}(Y_{k}),\,Q_{i'_1,j'_1,k'_1,i,j}(Y_{k})\in\rmR[Y_{k}],$ from which we deduce that they are all monomials. Moreover, as the $\Z^2$-degree of $q_{i,k,j}(\bfY)$ is $(0,0),$ we have then that \eqref{avobbe} becomes 
\begin{equation}\label{avobbbe}
Y_{k}= -r_{i,k,j}Y_{k}+\sum_{i'_1,j'_1,k'_1}r_{i'_1,j'_1,k'_1}Y_{k}^{A_{i'_1,j'_1,k'_1}}r_{i'_1,j'_1,k'_1,i,j}Y_{k}^{B_{i'_1,j'_1,k'_1,i,j}}
\end{equation}
for suitable  $r_{i,k,j},\,  r_{i'_1,j'_1,k'_1},\,r_{i'_1,j'_1,k'_1,i,j}\in\rmR,\,A_{i'_1,j'_1,k'_1},\,B_{i'_1,j'_1,k'_1,i,j}\in\N.$ Note that $r_{i,k,j}=q_{i,k,j}({\bf0}).$
\par
If $1+ r_{i,k,j}\neq0,$ then from \eqref{avobbbe} we deduce that there must be an admissible triplet $(i'_1,j'_1,k'_1)$ such that $r_{i'_1,j'_1,k'_1}\cdot r_{i'_1,j'_1,k'_1,i,j}\neq0$, and $A_{i'_1,j'_1,k'_1}+B_{i'_1,j'_1,k'_1,i,j}=1.$
From Lemma \ref{au1}, we deduce that  the case $B_{i'_1,j'_1,k'_1,i,j}=1$ cannot happen, as 
$r_{i'_1,j'_1,k'_1,i,j}Y_{k}^{B_{i'_1,j'_1,k'_1,i,j}}$ is one of the terms of $q_{i'_1,j'_1,k'_1,i,j}(\bfY),$ and the triplet $(i,j,k)$ is admissible.
\par
So, we must have $B_{i'_1,j'_1,k'_1,i,j}=0,$ which by Lemma \ref{au2} can only happen if not all of $\{\bfv_{i'_1},\,\bfv_{j'_1},\,\bfv_{k'_1}\}$ lie in the same cone. By Lemma \ref{yyavan}, we deduce that
$r_{i'_1,j'_1,k'_1,i,j}\in\langle p_{k_{s-1}}, p_{k_{s}}\rangle,$ $s$ being the index of the cone containing $\bfv_i.$ This implies that $1+r_{i,k,j}\in \langle p_{k_{s-1}}, p_{k_{s}}\rangle.$
Note that the latter also holds trivially if $1+r_{i,k,j}=0,$ so we can remove from now on our assumption on the value of this element. In addition, we must have  $r_{i,k,j}\neq0$ as otherwise we would have $1\in\langle p_{k_{s-1}},\,p_{k_s}\rangle,$ contradicting the hypothesis. So, the Theorem would have been proven already if the triplet $(i,k,j)$ is not admissible.
\par
To finish with our argument, we consider instead the coordinate $e_{(i,k)}$ in \eqref{sweet} and evaluate $Y_{u}\to0$ for $u\neq j$ to get
\begin{equation}\label{cer}
0=r_{i,k,j}Y_{j}+ \sum_{i'_1,j'_1,k'_1}q_{i'_1,j'_1,k'_1}(0,\ldots, Y_{j},\ldots, 0)q_{i'_1,j'_1,k'_1,i,k}(0,\ldots, Y_{j},\ldots, 0)
\end{equation}
instead of \eqref{29}. The $0$ at the left hand side of \eqref{cer} is due to the fact that if $(i,k,j)$ is admissible and has as one of its terms $Y_{k}e_{(i,j)},$ then we must have 
\begin{equation}\label{last1}
Y_{j}e_{(i,k)}=\mbox{lt}(\cs_{i,k,j})\succ Y_{k}e_{(i,j)}=\mbox{lt}(\cs_{i,j,k}),
\end{equation}
both terms appearing in $\cs_{i,k,j}.$
Note that If $q_{i,j,k,i,k}(\bfY),$ which appears in the expansion of $\cs_{i,j,k}$ in the coordinate $e_{(i,k)}$ were nonzero, then its $\Z^2$-degree would be equal to $\bfv_j$. After the evaluation  $Y_{u}\to0$ for $u\neq j$ it would become $rY_{j}$, which, thanks to Lemma \ref{au1}, would imply that $r=0.$ This justifies the $0$ in the left hand side of \eqref{cer}, and from there we deduce  as before 
\begin{equation}\label{samee}
-r_{i,k,j}Y_{j}=\sum_{i'_1,j'_1,k'_1}r_{i'_1,j'_1,k'_1}Y_{Y_{v_j}}^{A_{i'_1,j'_1,k'_1}}r_{i'_1,j'_1,k'_1,i,k}Y_{j}^{B_{i'_1,j'_1,k'_1,i,k}},
\end{equation}
for suitable  $r_{i,k,j},\,  r_{i'_1,j'_1,k'_1},\,r_{i'_1,j'_1,k'_1,i,j}\in\rmR,\,A_{i'_1,j'_1,k'_1},\,B_{i'_1,j'_1,k'_1,i,j}\in\N.$ 
Arguing as before, we deduce that  $-r_{i,j,k}\in \langle p_{k_{s-1}},\,p_{k_s}\rangle,$  (the index $s$ is the same as in the previous case, as it indexes the cone $C_s$ containing $\bfv_i$) which, combined with $1+r_{i,j,k}\in \langle p_{k_{s-1}},\,p_{k_s}\rangle,$ contradicts the hypothesis of the theorem. This concludes the proof.
\end{proof}

\begin{example}\label{phii1}
We will continue our running example with the computation of the kernel of $\varphi_1$ for the fan algebra $\cB_{\rmR, p} (3,2)$ from Remark~\ref{ff23}, (4). The map $$\varphi_1:  \rmR[\bfY]^5 \to \rmR [\bfY]$$ sends
$e_{i,j} \to S_{i,j},$ and we expect $2\binom{4}{3} =8$ generating relations in the kernel, due to the cardinality of $S_1$, as in Lemma~\ref{comb}. 
Sorted using the order induced by $\{ e_{1,3}, \ldots, e_{3,5} \}$ with the lexicographic order $Y_1 \prec \ldots \prec Y_5$ in mind, these are
\begin{equation}\label{cozik}
\begin{array}{ccl}
\cs_{1,3, 4} &=& Y_4e_{1,3} -Y_3e_{1,4}+Y_2e_{2,4}\\
\cs_{1,3, 5} &=& Y_5e_{1,3}-Y_3e_{1,5}+pY_4e_{2,4} +Y_2 e_{2,5}\\
\cs_{1, 4, 2} &=&  Y_2e_{1,4}-Y_1e_{2,4}-pe_{1,3}\\
\cs_{1,4, 5} &=& Y_5e_{1,4}-Y_4e_{1,5}+pe_{2,5}\\
\cs_{1,5, 2}&=& Y_2e_{1,5}-Y_1e_{2,5}-pY_4e_{1,4} \\
\cs_{1,5,3} &=& Y_3e_{1,5}-Y_1e_{3,5}-pY_4e_{2,4}-Y^2_4e_{1,4} \\
\cs_{2, 4, 5} &=& Y_5e_{2,4}-Y_4e_{2,5}+pe_{3,5}\\
\cs_{2, 5, 3} &=&Y_3e_{2,5}-Y_2e_{3,5} -Y^2_4e_{2,4}.
\end{array}
\end{equation}
For the first relation, we have the triplet $(1,3, 4) \in S_1$ and hence we need to consider the elements $e_{1,3}\to S_{1,3}=Y_1Y_3 -Y_2^2, e_{1,4} \to S_{1,4}= Y_1Y_4-pY_2$. The $S$-polynomial of $S_{1,3}, S_{1,4}$ is $Y_4(Y_1Y_3 -Y_2^2)-Y_3(Y_1Y_4-pY_2)=-Y_2(pY_3-Y_2Y_4)=-Y_2S_{2,4},$ so $Y_4e_{1,3}-Y_3e_{1,4}+Y_2e_{2,4}$ is in the kernel of $\varphi_1$, which is our first relation.

\end{example}

\section{Higher Syzygies}\label{hs}
From Theorem \ref{machiv} we deduce now the following exact sequence of $\rmR[\bfY]$-modules:
\begin{equation}\label{phi2}
\begin{array}{ccccccc}
\rmR[\bfY]^{2{M-1\choose 3}}& \stackrel{\varphi_2}{\to} &\rmR[\bfY]^{M-1\choose 2}&\stackrel{\varphi_1}{\to}& \rmR[\bfY]\ \stackrel{\varphi_0}{\to}&\cB_{\bff,\cF,\bfp}\to0\\
e_{(i,j,k)}&\mapsto&\cs_{i,j,k}&&&&\\
& & e_{(i,\,j)}&\mapsto & \bfS_{i,\,j},&&
\end{array}
\end{equation}
where we have made the identification $\rmR[\bfY]^{2{M-1\choose 3}}\simeq\bigoplus_{i,j,k}\rmR[\bfY]\cdot e_{(i,j,k)},$
the sum being over the admissible triplets $(i,j,k).$

To continue with a full resolution of \eqref{phi2}, we do a process similar to \eqref{sisi} recursively: 
we sort the elements in  $\bfS^{(2)}_{\bff,\cF, \bfp}$ with the lexicographic order in $\N^3,$ and given $1\leq i<j-1\leq n$ and $k<k'$ with $k,\,k'\in\{i+1,\ldots, n\}\setminus\{j-1,j\}$ we compute  the syzygy between $\cs_{i,j,k}$ and $\cs_{i,j,k'},$ which we denote by ${\bf S}_{i,j,k,k'}.$ Thanks to \eqref{juan}, we know that
\begin{equation}\label{arripoa}
{\bf S}_{i,j,k,k'}=Y_{{k'}}\cs_{i,j,k}-Y_{k}\cs_{i,j,k'}=\sum_{i_0,j_0,k_0}q_{i,j,k,k',i_0,j_0,k_0}(\bfY)\cs_{i_0,j_0,k_0},
\end{equation}
where the expression in the right hand side comes from applying the division algorithm to $Y_{{k'}}\cs_{i,j,k}-Y_{{k}}\cs_{i,j,k'}$ with respect to $\bfS^{(2)}_{\bff,\cF,\bfp}.$ We then define
\begin{equation}\label{gato}
\cs_{i,j,k,k'}=Y_{{k'}}e_{(i,j,k)}-Y_{{k}}e_{(i,j,k')}-\sum_{i_0,j_0,k_0}q_{i,j,k,k',i_0,j_0,k_0}(\bfY)e_{(i_0,j_0,k_0)}\in\mbox{\rm ker}(\varphi_2),
\end{equation}
and set $\bfS^{(3)}_{\bff,\cF,\bfp}$ the set of all these elements. As before, we will say that such a triplet $(i,j,k,k')$ is admissible. It is easy to verify that the cardinality of  $\bfS^{(3)}_{\bff,\cF,\bfp}$  is equal to the number of ordered $4$-tuples $(i,j,k,k')$ with $i<j-1,\,k<k'$ with $k,\,k'\in\{i+1,\ldots,M\}\setminus\{j-1,j\}.$  This number is $3{M-1\choose 4},$ thanks to Lemma \ref{comb} with $\ell=2.$ 

The following is be the generalization of \eqref{osit} to this case, its proof being straightforward by using Lemma \ref{tool} and \eqref{osit}.

\begin{lemma}\label{toool}
With notations as above, we have that 
$i\leq i_0,\,j_0, k_0\leq \max\{j,k,k'\}.$
\end{lemma}
Lemmas \ref{au1}, \ref{au2}, and \ref{yyavan} also have a straightforward generalization, which are given next.

\begin{lemma}\label{aau1}
With notation as in \eqref{gato}, if a  term of the form $rY_{t}$ with $r\in\rmR\setminus\{0\}$ appears in $q_{i,j,k,k'i_0, j_0, k_0}(\bfY),$ then the triplet $(i_0,j_0,k_0,t)$ is not admissible.
\end{lemma}
\begin{proof}
The presence of a term like that  implies that in the process of doing the division algorithm \eqref{preraw} we would have ran into a nonzero term of the form $rY_{t}Y_{{k_0}}e_{(i_0, j_0)}.$ From here, to produce the element $rY_{t}e_{(i_0,j_0,k_0)},$ we should  divide this term with  $\cs_{i_0,j_0,k_0}.$ It $(i_0, j_0, k_0,t)$ were admissible,  we could also divide it with $\cs_{i_0, j_0, t}.$ As $k_0<t,$ and due to the way we have designed the Division Algorithm in Remark \ref{usare} by choosing the maximal index in the family $\bfS^{(2)}_{\bff,\cF,\bfp},$ the division should be done with the latter, leading to a contradiction. This proves the claim. 
\end{proof}

\begin{lemma}\label{aau2}
For an admissible triplet $(i,j,k,k'),$ if  $\bfv_i,\,\bfv_j,\,\bfv_k$ and $\bfv_{k'}$ are all contained in one single cone of the fan, then $\cs_{i,j,k,k'}\in \langle Y_{1},\ldots, Y_{M}\rangle\rmR[\bfY]^{2{M-1\choose 3}}.$
\end{lemma}
\begin{proof}
In the conditions of the hypothesis, we have that both $\cs_{i,j,k}$ and $\cs_{i,j,k'}$ have all their coefficients in  $\langle Y_{1},\ldots, Y_{M}\rangle$ thanks to Lemma \ref{au2}. So, we actually have that the coefficients of $Y_{{k'}}\cs_{i,j,k}-Y_{k}\cs_{i,j,k'}$ belong to  $\langle Y_{1},\ldots, Y_{M}\rangle^2.$ The claim now follows straightforwardly as in the division algorithm over elements of  $\bfS^{(2)}_{\bff,\cF,\bfp},$ all the leading terms of these elements are one of the variables $Y_{u},\,u=1,\ldots, M.$
\end{proof}
\begin{lemma}\label{yyavanp}
With notation as above, if a nonzero constant term of the form $r\in\rmR$  appears in the Taylor expansion of one of the polynomials $q_{i,j,k,k',i_0, j_0, k_0}(\bfY)$ from \eqref{gato}, and $\bff$ is strict, then there is an index $s\in\{2,\ldots, \ell-1\}$ depending on $i_0$ such that $r\in\langle p_{k_{s-1}}, p_{k_{s}}\rangle.$
\end{lemma}
\begin{proof}
The proof follows again as in Lemma \eqref{yyavan}, by unravelling the Division Algorithm to produce 
the polynomial ${\bf S}_{i,j,k,k'}$ defined in \eqref{arripoa}. It is easy to see that all the leading terms during this Division Algorithm either belong to the ideal $\langle Y_1,\ldots, Y_M\rangle$ (and hence are not under the conditions of this hypothesis), or are multiple of some $p_{k_s}$ for $s\in\{1,\ldots, \ell-1\}.$ By using now Lemma \ref{yyavan}, it is easy to conclude that the two possible indices corresponding to $q_{i,j,k,k',i_0, j_0, k_0}(\bfY)$  are those determined by the cone where $\bfv_{i_0}$ lies. This concludes with the proof of the Lemma.
\end{proof}

 By applying now all these claims, we can reproduce mutatis mutandis the proof of Theorem \ref{machiv}.
\begin{theorem}\label{machiv2}
The set  $\bfS^{(3)}_{\bff,\cF, \bfp}$ is a minimal Gr\"obner basis of $\mbox{\rm syz}\big(\bfS^{(2)}_{\bff,\cF, \bfp}\big)$ for the induced order $\prec_{\cS^{(2)}}.$
If  $|\cF|$ is a proper convex polyhedral cone, $\bff$ is strict with respect to $\cF,$ none of the  $p_{k_i}$'s is  a zero divisor,  and  for $1\leq i\leq j\leq \ell-1,\,\langle p_{k_i}, p_{k_j}\rangle\neq\rmR,$ then it is also a minimal set of generators of this submodule.
\end{theorem}
\begin{proof}
The fact that this set is a minimal Gr\"obner basis follows directly from Theorem \ref{3713} and comparing the leading terms of all possible syzygies and those coming in $\bfS^{(3)}_{\bff,\cF, \bfp}.$

To prove minimal generation,  we proceed as in the proof of Theorem \ref{machiv}:  from an expression like
$\cs_{i,j,k,k'}=\sum_{i_1,j_1,k_1,k_1'}q_{i_1,j_1,k_1,k_1'}(\bfY)\cs_{i_1,j_1,k_1,k_1'},$ we pick the coordinate $e_{i,j,k}$ in both sides to get
\begin{equation}\label{3avobe}
Y_{k'}= \sum_{i_1,j_1,k_1,k_1'}q_{i_1,j_1,k_1,k_1'}(\bfY)q_{i_1,j_1,k_1,k_1',i,j,k}(\bfY) \end{equation}
(note that in this case there cannot be a term of the form $-q_{i,j,k',k}(\bfY)Y_{k'}$ as in \eqref{avobe} due to the fact that $(i,j,k',k)$ cannot be an admissible element as $k<k'$, so the proof gets actually simpler in this case). Now we set $Y_u\mapsto0$ for $u\neq k'$, and proceed mutatis mutandis the proof of Theorem \ref{machiv} by using now Lemmas \ref{toool}, \ref{aau1}, \ref{aau2}, and \ref{yyavanp}.
\end{proof}

\medskip
In general, for $s\geq3,$ given the map
$\varphi_s:\rmR[\bfY]^{s{M-1\choose s+1}}\to\rmR[\bfY]^{(s-1){M-1\choose s}}$ where the canonical basis of $\rmR[\bfY]^{s{M-1\choose s+1}}$ is denoted with $e_{i,j,k_1,\ldots, k_{s-1}}$ with $i<j-1,\,i<k_1<k_2<\ldots<k_{s-1},$ and no $k_t\in\{j-1,j\},$ and we have $e_{i,j,k_1,\ldots, k_{s-1}}\mapsto\cs_{i,j,k_1,\ldots, k_{s-1}}$ with $\mbox{\rm lt}(\cs_{i,j,k_1,\ldots, k_{s-1}})= Y_{k_{s-1}}e_{i,j,k_1,\ldots,k_{s-2}}\in\rmR[\bfY]^{(s-1){M-1\choose s}}.$ For a given $k_s>k_{s-1}$ such that $k_s\notin\{j-1,j\},$ we can proceed analogously as done in \eqref{arripoa} and \eqref{gato}, to produce $\cs_{i,j,k_1,\ldots, k_s}\in\mbox{ker}(\varphi_s)$  such that 
$\mbox{\rm lt}(\cs_{i,j,k_1,\ldots, k_{s}})= Y_{k_{s}}e_{i,j,k_1,\ldots,k_{s-1}}.$ Denote with $\bfS^{(s+1)}_{\bff,\cF,\bfp}$ the set of all such elements $\cs_{i,j,k_1,\ldots, k_s},$ which by Lemma \ref{comb} will have cardinality $(s+1){M-1\choose s+2}.$  We sort as usual this set by using the lexicographic order on $(i,j,k_1,\ldots, k_{s-1}).$

Straightforward generalizations of Lemmas \ref{toool}, \ref{aau1}, \ref{aau2}, and \ref{yyavanp} imply the following higher dimensional result.
\begin{theorem}\label{machiv3}
The set  $\bfS^{(s+1)}_{\bff,\cF, \bfp}$ is a minimal Gr\"obner basis of $\mbox{\rm syz}\big(\bfS^{(s)}_{\bff,\cF, \bfp}\big)$ for the induced order $\prec_{\cS^{(s)}}.$
If  $|\cF|$ is a proper convex polyhedral cone, $\bff$ is strict with respect to $\cF,$ none of the  $p_{k_i}$'s is  a zero divisor,  and  for $1\leq i\leq j\leq \ell-1,\,\langle p_{k_i}, p_{k_j}\rangle\neq\rmR,$ then it is also a minimal set of generators of this submodule.
\end{theorem}

Now we are in position to prove our main Theorem stated in the Introduction.
\begin{proof}[Proof of Theorem~\ref{mresult}]
The ring $\rmR[\bfY]$ is *-local with respect to the ideal $(\fm, \bfY)$, because $(\fm, \bfY)$ is a maximal ideal which is $\Z^2$-graded, and therefore it is also $\Z$-graded under the induced grading. According to the comments following Example~1.5.14 in~\cite{BH}, to have a minimal graded free resolution over $\rmR[\bfY]$, we need to check that  the kernel of $\varphi_0$ is contained in $(\fm, \bfY),$ and also that for each $j\geq1,\,\mbox{ker}(\varphi_j)\subset(\fm, \bfY)\rmR[\bfY]^{j{M-1\choose j+1}}.$ All these  follow from Theorems~\ref{mt}, \ref{machiv}, \ref{machiv2} and \ref{machiv3}.
\end{proof}

An immediate corollary of this result is
\begin{corollary}
\label{pd}
Under the notations of Theorem~\ref{mresult}, the projective dimension of $\cB_{\bff,\cF,\bfp}$ over $\rmR[\bfY]$ is $M-2$, if $(\rmR, \fm)$ is local, $\bff$ is strict with respect to $\cF,$ none of the $p_i$'s is a zero divisor in $\rmR,$ and all of them belong to $\fm$.

\end{corollary}

\begin{example}
\label{phii2}
We return to the Example~\ref{ff23} and continue Example~\ref{phii1} with the computation of the kernel of $\varphi_2: \rmR [\bfY]^{8} \to \rmR[\bfY]^ 6, $ given by $e_{i, j, k} \to \cs_{i,j, k}$, where $(i,j,k) \in S_1$.
We need to find the set $S_2= \{ (i,j,k,k'): 1 \leq i < j-1 \leq 5, k < k' \in \{i+1, \ldots, 5\} \setminus \{j-1, j\}\}$, see Definition~\ref{sets}. One gets $S_2 =\{ (1,3,4,5), (1,4,2,5), (1,5,2,3) \}$.


To compute $S$-polynomials here, it is better to write down explicitly ``matrix'' of the syzygies $\cs_{i,j,k}:$ we will encode in the following $8\times 6$ matrix the coordinates of the the syzygies in \eqref{cozik} in the basis $\{e_{1,3},e_{1,4},e_{1,5},e_{2,4},e_{2,5},e_{3,5}\}:$
$$\left(\begin{array}{cccccc}
Y_4&-Y_3&0&-Y_2&0&0\\
Y_5&0&-Y_3&p Y_4&Y_2&0\\
-p&Y_2&0&-Y_1&0&0\\
0&Y_5&-Y_4&0&p&0\\
0&-pY_4&Y_2&0&-Y_1&0\\
0&-Y_4^2&Y_3&-pY_4&0&-Y_1\\
0&0&0&Y_5&-Y_4&p\\
0&0&0&-Y_4^2&Y_3&-Y_2
\end{array}\right).
$$
This matrix has rank $5$, and computing explicitly a basis of its kernel, we get a basis of $3$ elements. From here it is easy to deduce the higher syzygies:
$$\begin{array}{ccl}
\cs_{1,3,4,5} &=&Y_5 e_{1,3,4}-Y_4 e_{1,3,5} +Y_3e_{1,4,5}-Y_2 e_{2,4,5}- p e_{2,5,3}\\
\cs_{1,4,2,5} &=& Y_5 e_{1,4,2}-Y_2 e_{1,4,5}+p e_{1,3,5}-Y_4 e_{1,5,2}+ p e_{1,5,3}+Y_1 e_{2,4,5}\\
\cs_{1, 5,2,3} &=& Y_3 e_{1,5,2}-Y_2 e_{1,5,3} +Y_1 e_{2,5,3}-p Y_4 e_{1,3,4} - Y_4^2 e_{1,4,2}.
\end{array}$$

Alternatively, we can compute these relations by hand. For $(1,3,4,5) \in S_2$ we need to consider the $S$-polynomial between $\cs_{1,3,4}$ and $\cs_{1,3,5}$ which is $Y_5 \cs_{1, 3, 4} - Y_4 \cs_{1, 3, 5} = -Y_3Y_5e_{1,4} +Y_2 Y_5 e_{2,4} +Y_3Y_4e_{1,5} -py_4^2e_{2,4} -Y_2Y_4e_{2,5}.$ The leading term in this expression is $-Y_3Y_5e_{1,4}$ so we can write  $Y_5 \cs_{1, 3, 4} - Y_4 \cs_{1, 3, 5}= -Y_3\cs_{1, 4, 5} + pY_3e_{2,5} +Y_2Y_5e_{2,4} -pY_4^2e_{2,4}-Y_2Y_4e_{2,5}$. The next leading term is now $-Y_2Y_4e_{2,5}$, so we need to use $\cs_{2, 4, 5}$ which leads to $Y_5 \cs_{1,3,4}-Y_4 \cs_{1,3,5} =Y_3\cs_{1,4,5}-Y_2 \cs_{2,4,5}- p \cs_{2,5,3} $. This in turn gives the relation:
$$\cs_{1,3,4,5} =Y_5 e_{1,3,4}-Y_4 e_{1,3,5} +Y_3e_{1,4,5}-Y_2 e_{2,4,5}- p e_{2,5,3}.$$ The rest can be obtained similarly.

We can now write the minimal resolution of $\cB_{\rmR, p}(3,2)$ over $\rmR[Y_1, \ldots, Y_5] = \rmR[\bfY]$ as

$$0 \to \rmR[\bfY]^3 \stackrel{\varphi_3}\to \rmR[\bfY]^8 \stackrel{\varphi_2}\to \rmR[\bfY]^6 \stackrel{\varphi_1}\to \rmR[\bfY] \stackrel{\varphi_0}\to \cB_{\rmR, p}(3,2) \to 0$$ where the matrix corresponding to $\varphi_2$ with respect to the bases
$\{e_{1,3},e_{1,4},e_{1,5},e_{2,4},e_{2,5},e_{3,5}\}$ and  $\{ \cs_{1,3, 4},\cs_{1,3, 5}, \cs_{1, 4, 2}, \cs_{1,4, 5}, \cs_{1,5, 2}, \cs_{1,5,3}, \cs_{2, 4, 5}, \cs_{2, 5, 3}\}$
is

$$\left(\begin{array}{cccccc}
Y_4&-Y_3&0&-Y_2&0&0\\
Y_5&0&-Y_3&p Y_4&Y_2&0\\
-p&Y_2&0&-Y_1&0&0\\
0&Y_5&-Y_4&0&p&0\\
0&-pY_4&Y_2&0&-Y_1&0\\
0&-Y_4^2&Y_3&-pY_4&0&-Y_1\\
0&0&0&Y_5&-Y_4&p\\
0&0&0&-Y_4^2&Y_3&-Y_2
\end{array}\right).
$$ and the matrix corresponding to $\varphi_3$ with respect to the bases $\{ \cs_{1,3, 4},\cs_{1,3, 5}, \cs_{1, 4, 2}, \cs_{1,4, 5}, \cs_{1,5, 2}, \cs_{1,5,3}, \cs_{2, 4, 5}\}$
and $\{ \cs_{1,3,4,5}, \cs_{1,4,2,5}, \cs_{1,5,2,3} \}$ is

$$\left(\begin{array}{cccccccc}
Y_5&-Y_4&0&-Y_3&0&0&-Y_2&p\\
0&p&Y_5&-Y_2& -Y_4& p & Y_1&0\\
-pY_4& 0 &-Y_4^2&0& Y_3 &-Y_2& 0&Y_1
\end{array}\right).
$$

\end{example}

\bigskip
\section{The case $|\cF|=\R^2$}\label{6}
We conclude this paper by showing some examples where the support of the fan is the whole plane, and our results do not hold in general in this case.
\par
We start with the simplest situation which is $\bfv_1=(1,0),\,\bfv_2=(0,1),\,\bfv_3=(-1,-1).$   Every pair of these vectors are adjacent, and there are three cones in this fan. So, the set $\cS_{\bff,\cF,\bfp}$ from Definition \ref{syzygy} is empty, and hence according to Theorem \ref{mt}, $\varphi_0$ should be an isomorphism.  However, it is straightforward to verify that $\langle Y_1Y_2Y_3-\bfp^{\bff(\bfv_1)+\bff(\bfv_2)+\bff(\bfv_3)}\rangle=\mbox{ker}(\varphi_0),$ so that statement cannot be true. Note also that in this example the only generator of the kernel does not have total degree $2$ as in our case of study..

Still for this example we have  a resolution like \eqref{eres}, so one may wonder whether a statement like Theorem \ref{mresult} holds, but with different maps $\varphi_j$ than those defined in the text.  In principle, the number of non adjacent relations changes as $\bfv_1$ and $\bfv_M$ are adjacent in this case. So there is already one syzygy less to consider of this type, but we do not know if this is the only difference.  For instance,  set $\rmR:=\Q[p_1, p_2, p_3]$ where each $p_i$ is an indeterminate, $1\leq i\leq 3,$  and use the data $\bfv_1=(1,0),\bfv_2=(1,1),\bfv_3=(0,1),\bfv_4=(-1,0),\bfv_5=(-1,-1),\bfv_6=(0,-1),$ and $f_1(x,y)=|y|,\,f_2(x,y)=|x-y|,$ and $f_3(x,y)=|x|.$ Note that ${\bf f}=(f_1,f_2,f_3)$ is strict. 
By computing explicitly the kernel of $\varphi_0$ we get the following $9$ non adjacent relations:
$$\begin{array}{l}
 Y_1Y_3-p_2^2Y_2,\
  Y_1Y_4-p_2^2p_3^2,\
Y_1Y_5 -p_3^2Y_6,\
Y_2Y_4-p_3^2Y_3, 
  Y_2Y_5- p_1^2p_3^2, \
  Y_2Y_6-     p_1^2Y_1,  \\
Y_3Y_5- p_1^2Y_4,\  
      Y_3Y_6- p_1^2p_2^2,
Y_4Y_6-p_2^2Y_5.
\end{array}$$
But note that the leading monomials of these polynomials in the lexicographic order are not always a product of two indeterminates. For instance if we use the lexicographic order $Y_1\prec\ldots\prec Y_6$, the leading term of of $Y_1Y_5 -p_3^2Y_6$ would be $-p_3^2Y_6$. If we reverse the lexicographic order, then we get that the leading term of $Y_2Y_6-p_1^2Y_1$ is actually $-p_1^2Y_1.$ This shows the further ideas are needed in order to explore this situation.

\bigskip
\section*{Acknowledgements}
All our computations and experiments were done with the aid of the softwares {\tt Macaulay 2}(\cite{mac}) and {\tt Mathematica} (\cite{math}). The last author thanks Sandra Spiroff for discussions related to the second section of the paper. Part of this work was done while the second author and third author were visiting Georgia State University and, respectively, Universitat de Barcelona. They are grateful for the support and working atmosphere provided by these institutions.

\end{document}